\documentclass[11pt,a4paper,english,reqno]{amsart}
\usepackage{babel}
\usepackage[latin1]{inputenc}
\usepackage[T1]{fontenc}
\usepackage{amsmath}
\usepackage{amsthm}
\usepackage{amsfonts}
\usepackage{indentfirst}
\usepackage{graphicx}
\usepackage{bm}

\usepackage{color}

\usepackage{xpicture}

\usepackage{amssymb}
\usepackage[mathscr]{eucal}

\newtheorem{teo}{Theorem}[section]
\newtheorem{de}[teo]{Definition}
\newtheorem{pro}[teo]{Proposition}
\newtheorem{cor}[teo]{Corollary}
\newtheorem{rem}[teo]{Remark}
\newtheorem{lem}[teo]{Lemma}

\newtheorem{conj}[teo]{Conjecture}
\newtheorem*{theorem*}{Theorem}

\newcommand{\co}{{\mathcal O}}

\newcommand{\gp}{\mathbb{P}}

\newcommand{\agot}{{\mathfrak a}}

\renewcommand{\int}{{\rm int}}

\newcommand{\betabarra}{\bar{\beta}}

\title[Minimal plane valuations]{Minimal plane valuations}

\author{Carlos~Galindo, Francisco~Monserrat \and Julio-Jos\'e~Moyano-Fern\'andez}
\curraddr{\texttt{Carlos Galindo and Julio Jos\'e Moyano-Fern\'andez:} Instituto Universitario de Ma\-te\-m\'a\-ti\-cas y Aplicaciones de
Castell\'on and Departamento de Matem\'aticas, Universitat Jaume I,
Campus de Riu Sec. s/n, 12071 Castell\'{o} de la Plana (Spain).} \email{
galindo@mat.uji.es; moyano@uji.es} \curraddr{\texttt{Francisco Monserrat:} Instituto Universitario de
Matem\'atica Pura y Aplicada, Universidad Polit\'ecnica de Valencia,
Camino de Vera s/n, 46022 Valencia (Spain).}
\email{framonde@mat.upv.es}
\date{}
\thanks{The authors were partially supported by the Spanish Government Ministerio de Econom\'ia, Industria y Compe\-ti\-ti\-vi\-dad/FEDER, grants MTM2012-36917-C03-03, MTM2015-65764-C3-2-P and MTM2016-81735-REDT, as well as by Universitat Jaume I, grant P1-1B2015-02.}
\date{}

\begin{document}

\maketitle

\begin{abstract}
We consider the value $\hat{\mu} (\nu) = \lim_{m \rightarrow \infty} m^{-1} a(mL)$, where $a(mL)$ is the last value of the vanishing sequence of $H^0(mL)$ along a divisorial or irrational valuation $\nu$ centered at $\mathcal{O}_{\mathbb{P}^2,p}$, $L$ (respectively, $p$) being a line (respectively, a point) of the projective plane $\mathbb{P}^2$ over an algebraically closed field. This value contains, for valuations, similar information as that given by Seshadri constants for points. It is always true that $\hat{\mu} (\nu) \geq \sqrt{1 / \mathrm{vol}(\nu)}$ and minimal valuations are those satisfying the equality. In this paper, we prove that the Greuel-Lossen-Shustin Conjecture implies a variation of the Nagata Conjecture involving minimal valuations (that extends the one stated in \cite{d-h-k-r-s} to the whole set of divisorial and irrational valuations of the projective plane) which also implies the original Nagata's conjecture. We also provide infinitely many families of minimal very general  valuations with an arbitrary number of Puiseux exponents, and an asymptotic result
that can be considered as evidence in the direction of the above mentioned conjecture \cite{d-h-k-r-s}.
\end{abstract}

\section{Introduction}

In \cite{b-k-m-s} an analogue of the Seshadri constant for real valuations $\nu$ is introduced. We denote this quantity by $\hat{\mu}(\nu)$. This value turns out to be very hard to compute. Considering rank one valuations of the projective plane and denoting $\mathrm{vol}(\nu)$ the volume of the valuation $\nu$, it holds that $\hat{\mu}(\nu) \geq \sqrt{1/\mathrm{vol}(\nu)}$ and $\nu$ is said to be minimal whenever the equality is satisfied. Minimal valuations of the complex projective plane whose dual graphs are determined by a unique real number are studied in \cite{d-h-k-r-s}, where the authors conjecture that every very general valuation as above such that $1/\mathrm{vol}(\nu) \geq 9$ is minimal. There, it is also proved that this conjecture implies Nagata's conjecture and is implied by the Greuel-Lossen-Shustin Conjecture \cite{greuel}. The present paper has two goals. Firstly, following the ideas in \cite{d-h-k-r-s}, we extend the above conjecture and implications to {\it any} divisorial or irrational valuation of the projective plane. Our development works for any algebraically closed field, and we only fall on complex numbers when the Greuel-Lossen-Shustin Conjecture is needed. Secondly, we prove an asymptotic result which gives support to the conjecture in \cite{d-h-k-r-s}.  In addition, we give an upper bound of $\hat{\mu}$ for a wide class of very general (see Definition \ref{vgd}) divisorial valuations of the projective plane, including those considered in \cite{d-h-k-r-s}. Also,  we provide infinitely many families of minimal very general valuations, with arbitrarily many Puiseux exponents, and determined only by their dual graphs. The three previous results are achieved with the help of valuations, non-positive at infinity.

According to some authors, valuations were introduced by Dedekind and Weber in 1882 to give a rigorous presentation of Riemann surfaces. It seems that the first axiomatic definition is due to K\"ursch\'{a}k in 1912; since then they have been used in several mathematical areas. In the forties and fifties, Zariski and Abhyankar applied valuation theory to resolution of singularities of algebraic varieties \cite{ja21, zar1, ja23, ja1, ja2} and, although Hironaka proved resolution of singularities in characteristic zero without using valuations, they seem to be the preferred tool to try to get resolution in positive characteristic \cite{tei}. Valuations are  essentially local objects which are expected to be useful for proving local uniformization. Recently, they also appear as a tool in global geometry and, at least certain classes of valuations
that reproduce the behavior of plane curves at infinity \cite{ja4}, allow us to determine  the cone of curves, the nefness of certain divisors or the finite generation of the Cox ring  of their attached surfaces \cite{cox}. Some  of the results in this paper will use this class of valuations.

Although there is no complete classification of valuations, valuations of quotient fields of  two-dimensional regular local rings $(R,\mathfrak{m})$, centered at $R$, were classified by Spivakovsky \cite{ja20} (see also \cite{ja15,fj}). We call those valuations plane valuations and they are in bijective correspondence with simple sequences of point blowing-ups starting with the blow-up of Spec$\,R$ at $\mathfrak{m}$. Spivakovsky's classification is based on the structure of the dual graphs of the valuations, which are defined by the exceptional divisors of the above mentioned sequences. These dual graphs can be determined by sequences of values called the Puiseux exponents of the valuations. The above mentioned classification contains five types of valuations but we are only interested in two of them, the so-called divisorial and irrational valuations. Puiseux exponents for them are of the form $\{\beta'_i\}_{i=0}^{g+1}$ and they are rational numbers except for the last one in the irrational case, which satisfies $\beta'_{g+1} \in \mathbb{R} \setminus \mathbb{Q}$.

Both divisorial and irrational plane valuations are Abhyankar valuations and their respective triples (rank, rational rank, transcendence degree) are $(1,1,1)$ and $(1,2,0)$. We recall that Abhyankar valuations admit local uniformization in any characteristic \cite{knaf} and satisfy a strengthened form of Izumi's Theorem \cite{e-l-s} (see also \cite{bou}). In \cite{e-l-s} the authors introduce the concept of volume of a rank one valuation $\nu$, defined as
\[
\mathrm{vol}(\nu) =\limsup_{\alpha \rightarrow \infty} \frac{\mathrm{length} (R / \mathcal{P}_\alpha)}{\alpha^n/n!},
\]
where $\dim R =n$ and
$\mathcal{P}_\alpha = \{f \in R | \nu(f) \geq \alpha\} \cup \{0\}$.
There, it is also proved that
\[
\mathrm{vol} (\nu) = \lim_{\alpha \rightarrow \infty} \frac{e(\mathcal{P}_\alpha)}{\alpha^n},
\]
where $e(\mathcal{P}_\alpha)$ is the multiplicity of the valuation ideal. In our plane case, the normalized volume of a valuation can be easily computed as $\betabarra_0^2/\betabarra_{g+1}$, where $\betabarra_0$ and $\betabarra_{g+1}$ are the first and last elements of the so-called sequence of maximal contacts of $\nu$ (see Sections  \ref{formulas} and \ref{volumen}).

Seshadri constants are objects that measure local positivity and were considered by Demailly for studying the Fujita's conjecture \cite{5Nach}. Analogous objects for ideal sheaves were introduced in \cite{c-ein-l} and named $s$-invariants. Rationality of those constants is an interesting issue which is  open even in the two-dimensional case. In \cite{b-k-m-s}, the authors introduce  the vanishing sequence along a valuation $\nu$ of $H^0(B)$, $B$ being a line bundle over a normal projective variety; we denote $\hat{\mu}_B (\nu) = \lim_{m \rightarrow \infty} m^{-1} a(mB)$, where $a(mB)$ is the last value for the line bundle $mB$ of the above mentioned sequence. This value is close to the $s$-invariant attached to the valuation and basically contains the same information for valuations as Seshadri constants for points; in addition, when applied to smooth projective surfaces, the blowing-up at one point and ample line bundles, it coincides with that constant in case of being irrational \cite{bauer} (see also \cite{Nach}).

In \cite{b-k-m-s}, it is proved that the value $\hat{\mu}_D(\nu)$,  for a real valuation $\nu$ centered at a point of an algebraic surface $X$ and an ample divisor $D$,  satisfies the inequality
\begin{equation}
\label{MIN}
\hat{\mu}_D(\nu) \geq \sqrt{D^2/\mathrm{vol}(\nu)}.
\end{equation}
In this paper, we are interested in divisorial and irrational plane valuations of the fraction field $K$ of $R=\mathcal{O}_{\mathbb{P}^2,p}$ centered at $R$, $\mathbb{P}^2 = \mathbb{P}^2_k$ being the projective plane over an algebraically closed field $k$ of arbitrary characteristic. Valuations of this type and satisfying, for a line $D$, the equality in (\ref{MIN}) are named {\it minimal} and their study is the objective of this paper. Let $(u,v)$ be local coordinates at $p$. The quantity $\hat{\mu}_D(\nu) = \hat{\mu}(\nu)$ is defined as follows. Let $\mu_d (\nu) = \max \{\nu(f) \;| \; f \in k[u,v], \; \mathrm{deg}(f) \leq d\}$ and then
 \[
 \hat{\mu}(\nu) = \lim_{d \rightarrow \infty} \frac{\mu_d (\nu)}{d}.
 \]

Recently in \cite{d-h-k-r-s}, minimal divisorial and irrational valuations (of $K$ centered at $R$) whose dual graphs are given by a unique positive rational or non-rational real value, and where $k=\mathbb{C}$, have been treated (see Remark \ref{quasi} for a description in our terminology of the mentioned specific case studied in \cite{d-h-k-r-s}). There, a variation of the Nagata Conjecture (implying that conjecture) in terms of minimality of very general valuations is given. Specifically, the conjecture  asserts that if $\nu$ is a very general divisorial or irrational valuation as above, whose inverse normalized volume is not less than 9, then $\nu$ is minimal. Notice also that the value $\hat{\mu}(\nu)$ is essential for computing the Newton-Okounkov bodies of flags given by divisors (and points on them) defining divisorial valuations as mentioned \cite{cil}.

This paper has two main goals. First, we extend to {\it any} divisorial and irrational valuation (of $K$ centered at $R$) the result in \cite{d-h-k-r-s} that asserts that the mentioned variation of the Nagata Conjecture (Conjecture \ref{conj2}) is true under the assumption that the so-called Greuel-Lossen-Shustin Conjecture (Conjecture \ref{conj1}) holds. More specifically, we prove
\begin{theorem*}[A]
Under the assumption that $k=\mathbb{C}$, the Greuel-Lossen-Shustin Conjecture implies that any very general divisorial or irrational valuation of the projective plane $\nu$ such that $[{\rm vol}(\nu^N)]^{-1}\geq 9$ is minimal, where $\nu^N$ is the normalization of $\nu$, obtained by dividing by $\nu(\mathfrak{m})$.
\end{theorem*}
To this purpose, we consider a classical tool for the treatment of curves and valuations in the positive characteristic case, the Hamburger-Noether expansions \cite{campillo, d-g-n}. Section \ref{PRE} of the paper provides the information we need with respect to the Hamburger-Noether expansions and the most useful invariants of our valuations and Section \ref{metric} proves the technical results to be used to show that Conjecture \ref{conj1} implies Conjecture \ref{conj2}. Here, our main result, Proposition \ref{proposiciondos}, extends, by a very different procedure, the result  on the continuity of the function $\hat{\mu}$ given in \cite{d-h-k-r-s} to arbitrary valuations. Our definition of $\hat{\mu}$ does not depend on the characteristic of the field $k$ and our reasoning does not depend on the characteristic, but our conclusion on the variation of the Nagata Conjecture must be over the complex field $\mathbb{C}$ because Conjecture \ref{conj2} was stated over $\mathbb{C}$.

The second goal of the paper is developed in Sections \ref{VERY} and \ref{laseis}. It is valid for arbitrary characteristic and, as a fundamental tool, we use divisorial valuations, non-positive at infinity, which were introduced in \cite{cox} (see Definition \ref{NON}). Here, our first main result, Theorem \ref{zxz} in the paper, provides an upper bound of $\hat{\mu}$ for a wide class of very general divisorial plane valuations.
\begin{theorem*}[B]
Let $\nu$ be a very general divisorial valuation of the projective plane (which is not the $\mathfrak{m}$-adic one), set $\{\betabarra_i\}_{i=0}^{g+1}$ its sequence of maximal contact values and assume that $\betabarra_1^2\geq \betabarra_{g+1}$. Then $\hat{\mu}(\nu)\leq \betabarra_1$. If, in addition, the set $A_{\nu}:=\{a\in \mathbb{Z}\mid 1\leq a\leq \lfloor \betabarra_1/\betabarra_0 \rfloor \mbox{ and } a^2\betabarra_0^2\geq \betabarra_{g+1} \}$ is not empty, then $\hat{\mu}(\nu)\leq \betabarra_0\cdot \min(A_{\nu})$.
\end{theorem*}
Next we state our second main result (see Theorem \ref{rock} for the precise details).

\begin{theorem*}[C] Fix any divisorial valuation, different from the $\frak{m}$-adic one. Let $\Gamma$ be its dual graph. Dual graphs of minimal very general valuations can be obtained from $\Gamma$ by adjoining chains with a suitable number of vertices (corresponding to divisors obtained by blowing-up free points) at the beginning and at the end of $\Gamma$.
\end{theorem*}
Notice that this result provides (infinitely many) families of minimal very general divisorial valuations (of $K$ centered at $R$). These families are characterized only by the dual graphs of their valuations.

The described minimal valuations cannot be defined by satellite divisors;
however we prove the following asymptotic result.
\begin{theorem*}[D]
For each positive real number $t$, denote by $\nu_t$ a very general quasi-monomial valuation (in the sense of the definition given in \cite{d-h-k-r-s}) with normalized volume $1/t$. Then
\[
\lim_{t \rightarrow \infty} \frac{\hat{\mu}(\nu_t^N)}{\sqrt{t}} =1.
\]
\end{theorem*}
This result, Theorem \ref{teo57} in the paper, can be considered as evidence in the direction of Conjecture \ref{conj2} (more specifically, of the restricted form of the conjecture stated in \cite{d-h-k-r-s}).

\section{Preliminaries} \label{PRE}

\subsection{Plane valuations}

\label{pdv}

Let $F$ be a field, $G$ a totally ordered abelian group, and set $F^* :=F \setminus\{0\}$. A valuation of $F$ is a surjective map $\nu: F^* \rightarrow G$ such that
$$
\nu(f+g) \geq \min \{\nu(f), \nu(g)\} \ \  \ \mbox{and} \ \ \  \nu(fg) = \nu (f) + \nu (g)
$$
for $f,g \in F^*$. The group $G$ is called the value group of the valuation $\nu$.

Assume that $F$ is the quotient field of a local regular ring $(R,\mathfrak{m})$ and that the residue field $k:=R/\mathfrak{m}$ is contained in $R$. We write $R_\nu$ for the valuation ring of $\nu$, that is
$$
R_\nu = \{f \in F^* | \nu(f) \geq 0 \} \cup \{0\}.
$$
This ring is also a local ring whose maximal ideal is $\mathfrak{m}_\nu = \{f \in F^* | \nu(f) >0 \} \cup \{0\}$. The valuation $\nu$ is said to be centered at $R$ whenever $R \cap \mathfrak{m}_\nu = \mathfrak{m}$.

There are several invariants of $\nu$ introduced by Zariski \cite{zar1,ja23} for the classification of plane valuations---a certain kind of valuations which we will define later: the {\it rank} $\mathrm{rk} (\nu)$ of $\nu$, which is the Krull dimension of the ring $R_\nu$, and the {\it rational rank} of $\nu$, written r.rk($\nu$), which is the dimension of the $\mathbb{Q}$-vector space $G \otimes_\mathbb{Z} \mathbb{Q}$. Moreover, if one considers the residue field $k_\nu:=R_\nu /\mathfrak{m}_\nu$, then the transcendence degree of the field extension $k_\nu / k$ is called the transcendence degree of $\nu$, denoted by $\mathrm{tr.deg} (\nu)$.

We assume that the field $k$ is algebraically closed. It is a result by Abhyankar \cite{ja2} that
\begin{equation} \label{eqn:Zinequalities}
{\rm rk}(\nu)+{\rm tr.deg}(\nu) \leq {\rm r.rk}(\nu)+{\rm
tr.deg}(\nu) \leq \dim(R).
\end{equation}
If the second inequality turns out to be an equality, then $G$ is isomorphic to $\mathbb{Z}^{{\rm r.rk}(\nu)}$ with some ordering, and when the equality ${\rm rk}(\nu)+{\rm tr.deg}(\nu)=\dim R$ holds, then $G$ is isomorphic to $\mathbb{Z}^{{\rm r.rk}(\nu)}$ with the lexicographical ordering.

For a valuation $\nu$ as above, we can attach to the ring $R$ the monoid
\[
S=\{\nu(f)|f \in R \setminus \{0\}\} \cup \{0\},
\]
which is called the {\it value semigroup} of $\nu$ associated to $R$. Ideals in $R$ which are contraction of ideals in $R_{\nu}$ are the so-called {\it valuation ideals}. They can be characterized as those ideals $\agot$ in $R$ of the form
\[
\{f \in R  | \nu(f) \geq \nu (\agot) \} \cup \{0\},
\]
where $\nu(\agot):=\min \{\nu (f) | f  \in \agot\setminus \{0\}\}$. In particular, when $G= \mathbb{Z}$, for a nonnegative integer $\alpha$, the ideal ${\mathcal P}_{\alpha} = \{f \in R \setminus \{0\} | \nu(f) \geq \alpha\} \cup \{0\}$ is a valuation ideal in $R$.
\medskip

A valuation $\nu$ is called {\it divisorial} if both $\mathrm{rk} (\nu) =1$ and $\mathrm{tr.deg} (\nu) = \dim R -1 $. Given a birational morphism $\pi: X \rightarrow \mathrm{Spec} R$ and an irreducible component $D$ of $\pi^{-1} (\mathfrak{m})$ such that $\mathcal{O}_{X,D}$ is a discrete valuation ring, the map
\begin{equation}\label{ooo}
\nu(f) = c \cdot\mathrm{ord}_D (f), \ \   \  f \in F
\end{equation}
(for some non-zero constant $c$) defines a divisorial valuation; in fact, all divisorial valuations centered at $R$ are of this form \cite[Remark 2.7]{ja20}.  Notice that for a (commutative, with unit) ring $T$, an ideal $\mathfrak{n}$ of $T$ and an element $h \in T$, one may define the {\it multiplicity} of $h$ at $\mathfrak{n}$ as
$$
\mathrm{mult}_\mathfrak{n} (h)= \max \{ s \in \mathbb{N} \cup \{0\} | h \in \mathfrak{n}^s\}.
$$
In addition, consider an integral scheme $X$ with function field $F$, and an integral subscheme $D$ such that $\mathcal{O}_{X,D}$ is a regular local ring with maximal ideal $\mathfrak{n}$; for an element $f=a/b \in F$, with $a,b \in R$, the {\it order} of $f$ at $D$ is defined to be
$$
\mathrm{ord}_D (f) = \mathrm{mult}_\mathfrak{n} \pi^* a - \mathrm{mult}_\mathfrak{n} \pi^* b.
$$

Divisorial valuations as in (\ref{ooo}) associated with different non-zero constants $c$ are called \emph{equivalent}. Along this paper, for each equivalence class of divisorial valuations, we will only consider two of them: the valuation $\nu$ corresponding to $c=1$ (this will be the default case) and its \emph{normalization} $\nu^{N}$, where $c=1/\nu(\mathfrak{m})$, $\nu(\mathfrak{m})$ being the minimum of the values $\nu(\varphi)$ for $\varphi\in \mathfrak{m}\setminus \{0\}$.

Let us restrict ourselves now to Noetherian regular local rings $(R,\mathfrak{m})$ of dimension 2. In this situation, plane valuations play a central role. A \emph{plane valuation} is a valuation of a field $K$ which is the fraction field of such an $R$, and centered at $R$. As mentioned above, Zariski gave an algebraic classification for plane valuations based on the values of the invariants involved in (\ref{eqn:Zinequalities}) may take. In this work, we are mainly interested in two types of plane valuations: divisorial and irrational valuations. A plane valuation $\nu$ is called \emph{divisorial} if, according to the general definition, satisfies the equalities
\[
\mathrm{rk} (\nu) = \mathrm{r.rk} (\nu) = \mathrm{tr.deg} (\nu) =  1;
\]
on the other hand, if $\nu$ is such that $\mathrm{rk} (\nu) = 1$, $\mathrm{r.rk} (\nu) = 2$ and $\mathrm{tr.deg} (\nu) =  0$, then it is called \emph{irrational}.
Again Zariski found a fruitful geometric viewpoint in dealing with plane valuations:

\begin{teo} \label{b}  There is a one-to-one correspondence between the set of
plane valuations of $K$ centered at $R$ (up to equivalence) and the set of simple
sequences of point blowing-ups of the scheme {\rm Spec} $R$.
\end{teo}

Theorem \ref{b} means that we may associate a valuation $\nu$ with a sequence
\begin{equation}
\label{uno} \pi: \cdots \longrightarrow X_{N+1}
\stackrel{\pi_{N+1}}{\longrightarrow} X_{N} \longrightarrow \cdots
\longrightarrow X_{1} \stackrel{\pi_{1}}{\longrightarrow} X_{0}=X=
{\rm Spec} \;R,
\end{equation}
where $\pi_1$ is the blowing-up of $X_0$ at $p_1= \mathfrak{m}$ and, for $i\geq 1$, $\pi_{i+1}$  the blowing-up of $X_i$ at the unique closed point $p_{i+1}$ of the exceptional divisor $E_i$, defined by $\pi_i$, for which $\nu$ is centered at the local ring $R_i:=\mathcal{O}_{X_i,p_{i+1}}$.

The sequence (\ref{uno}) of blowing-ups needs not be finite and defines a sequence (or confi\-guration) of infinitely near points which will be denoted by $\mathcal{C}_{\nu}:=\{p_i\}_{i\geq 1}$. A point $p_i$ is said to be {\it proximate} to $p_j$, denoted by $p_i \rightarrow p_j$, whenever $i>j$ and $p_i$ belongs either to $E_{j}$ or to the strict transform of $E_{j}$ on $X_{i-1}$. A  point $p_i$ is called {\it satellite} if there exists $j < i-1$ satisfying $p_i \rightarrow p_j$; otherwise, it is called {\it free}. Divisors obtained by blowing-up free (respectively, satellite) points are also called free (respectively, satellite).

The {\it dual graph} of a valuation $\nu$ as above is a labeled tree (with infinitely many vertices if ${\mathcal C}_{\nu}$ is infinite) where each vertex represents an exceptional divisor  appearing in the sequence of blowing-ups (\ref{uno}) and two vertices are joined whenever their corresponding divisors meet. Each vertex is labeled with the number of blowing-ups needed to create the corresponding divisor. Equivalently, some authors label each vertex with the self-intersection of the divisor. When the valuation $\nu$ is divisorial, we add an arrow to the vertex that represents the defining divisor of $\nu$.

\subsection{Hamburger-Noether expansions}\label{s:hne}
Puiseux series are not suitable for the treatment of curve singularities in positive characteristic. This treatment can be addressed by using Hamburger-Noether expansions, which work with zero and positive characteristic (see 
\cite{campillo}).

In this paper we are interested in plane valuations. We start by briefly explaining the concept of Hamburger-Noether expansion of a plane valuation $\nu$, from which we will read off the invariants of the valuation. Starting from a regular system of parameters $(u,v)$ for $R$, Hamburger-Noether expansions allow us to classify plane valuations according to the expansion's shape and explicitly determine a regular system of parameters for the rings $R_1,R_2,\ldots, R_N, \ldots$ given by the blow-ups sequence $\pi$ (\ref{uno}). Next we describe this concept;  an alternative version for divisorial valuations can be found in
\cite{ja15}.

To begin with, as mentioned, choose a regular system
of parameters $(u,v)$ for the ring $R$, and assume that $\nu(u) \leq \nu(v)$.
As a first step, we can choose an element $a_{01} \in k$ such that
\[
(u_1 := u, v_1 := (v/u) - a_{01} )
\]
is a regular system of parameters for the ring $R_1$, and therefore $v=a_{01}u+uv_1$.

In the second step there are two possibilities: either $\nu(v_1) \geq \nu(u)$, and we can then choose an element $a_{02} \in k$ such that
\[
(u_2=u, v_2=(v_1/u) - a_{02})
\]
is a regular system of parameters for the ring $R_2$ in such a way that
$$
v=a_{01}u+u(a_{02}u + uv_2)=a_{01}u+a_{02}u^2 + u^2v_2,
$$
or $\nu(v_1)<\nu(u)$, and we stop the algorithm: this is the case $\nu(v)=\nu(u)$.

We keep doing the same
procedure until we obtain
\[
v=a_{01}u+a_{02}u^{2}+ \cdots +
a_{0h}u^{h}+u^{h}v_{h},
\]
where either $\nu(u) > \nu (v_h)$ or
$\nu(v_h)=0$, or
\[
v=a_{01}u+a_{02}u^{2}+ \cdots + a_{0h}u^{h}+
\cdots,
\]
with infinitely many steps. In the last two cases we are done, and we get the Hamburger-Noether expansion for $\nu$. Moreover, $R_{\nu} =
R_h$ when $\nu(v_h)=0$. Otherwise, set $w_1 := v_h$ so that $(w_1,u)$ is a regular system of parameters for $R_h$ with $\nu (w_1)<\nu(u)$, and we start with the process again.

This procedure can continue indefinitely or we can
obtain a last equality. In any case we have attached to $\nu$ a set of
expressions of the form given in Figure \ref{hne}, which is called the Hamburger-Noether expansion of the valuation $\nu$ in the regular system of parameters $(u,v)$ of the ring $R$.

\begin{figure}[h]
\[
\begin{array}{lccl}
&w_{-1}=v & = & a_{01}u+a_{02}u^{2}+ \cdots + a_{0h_{0}}u^{h_{0}}+u^{h_{0}}w_{1} \\
&w_0=u & = & w^{h_{1}}_{1}w_{2} \\
&\vdots & \nonumber & \vdots \\
&w_{s_{1}- 2} & = & w^{h_{s_{1}-1}}_{s_{1}-1}w_{s_{1}}\\
&w_{s_{1}-1} & = & a_{s_{1}k_{1}}w^{k_{1}}_{s_{1}} + \cdots
 +a_{s_{1}h_{s_{1}}}w^{h_{s_{1}}}_{s_{1}}+w^{h_{s_{1}}}_{s_{1}}w_{s_{1}+1}
 \\
& \vdots & \nonumber & \vdots \\
&w_{s_{g}- 1} & = & a_{s_{g}k_{g}}w^{k_{g}}_{s_{g}} + \cdots
+a_{s_{g}h_{s_{g}}}w^{h_{s_{g}}}_{s_{g}}+w^{h_{s_{g}}}_{s_{g}}w_{s_{g}+1} \\
&\vdots & \nonumber & \vdots \\
&w_{i- 1} & = & w^{h_{i}}_{i}w_{i+1} \\
&\vdots & \nonumber & \vdots \\
&(w_{z-1} & = & w^{\infty}_{z}).
\end{array}
\]
\caption{Hamburger-Noether expansion of a plane valuation}
\label{hne}
\end{figure}

The nonnegative integers $s_0 = 0, s_1, \ldots , s_g$ correspond to rows with
some nonzero elements $a_{s_{j}l} \in k$: these rows are called \emph{free}, and are given by
free blowing-up points. Notice that $g \in \mathbb{N}\cup
\{ \infty \} $ and $k_j = \min \{n \in \mathbb{N} \, | \, a_{s_{j}n} \not =
0 \} $.

Plane valuations can be classified in five types according to the shape of their Hamburger-Noether expansions \cite[Section 1.4]{d-g-n}.
Two of the classes are the divisorial and irrational valuations:\\

{\it $\diamond$ Divisorial valuations} have a finite sequence (\ref{uno}), and therefore the associated Ham\-burger-Noether expansion is finite; the last row of the expansion has the form
\begin{equation}
\label{DVI}
 w_{s_{g}- 1}  = a_{s_{g}k_{g}}w^{k_{g}}_{s_{g}}+
\cdots
+a_{s_{g}h_{s_{g}}}w^{h_{s_{g}}}_{s_{g}}+w^{h_{s_{g}}}_{s_{g}}w_{s_{g}+1},
\end{equation}
with $ g< \infty $, $ h_{s_{g}}  < \infty $, $w_{s_{g}+1} \in
R_{\nu}$ and $\nu(w_{s_{g}+1}) =0$.

\begin{figure}[h]
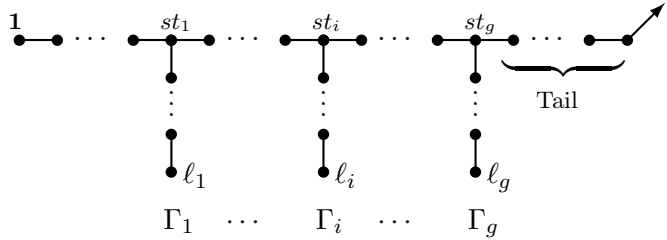


\begin{center}
\setlength{\unitlength}{0.5cm}%
\begin{Picture}(0,0)(20,8)
\thicklines


\xLINE(0,6)(1,6)
\Put(0,6){\circle*{0.3}}
\Put(1,6){\circle*{0.3}}
\put(1,6){$\;\;\ldots\;\;$}
\xLINE(3,6)(4,6)
\Put(3,6){\circle*{0.3}}
\Put(4,6){\circle*{0.3}}
\xLINE(4,6)(4,5)
\Put(4,5){\circle*{0.3}}
\Put(3.9,4){$\vdots$}
\xLINE(4,3.5)(4,2.5)
\Put(4,3.5){\circle*{0.3}}
\Put(4,2.5){\circle*{0.3}}
\Put(3.8,1){$\Gamma_1$}

\put(4.3,2.2){$\ell_1$}

\Put(3.7,6.3){\footnotesize $st_1$}
\Put(-0.3,6.3){\footnotesize $\mathbf{1}$}



\xLINE(4,6)(5,6)
\Put(4,6){\circle*{0.3}}
\Put(5,6){\circle*{0.3}}
\Put(5,6){$\;\;\ldots\;\;$}
\xLINE(7,6)(8,6)
\Put(7,6){\circle*{0.3}}
\Put(8,6){\circle*{0.3}}
\xLINE(8,6)(8,5)
\Put(8,5){\circle*{0.3}}
\Put(7.9,4){$\vdots$}
\xLINE(8,3.5)(8,2.5)
\Put(8,3.5){\circle*{0.3}}
\Put(8,2.5){\circle*{0.3}}
\Put(7.8,1){$\Gamma_i$}

\put(8.3,2.2){$\ell_i$}

\Put(7.7,6.3){\footnotesize $st_i$}

\Put(5,1){$\;\;\cdots\;\;$}


\xLINE(8,6)(9,6)
\Put(8,6){\circle*{0.3}}
\Put(9,6){\circle*{0.3}}
\put(9,6){$\;\;\ldots\;\;$}
\xLINE(11,6)(12,6)
\Put(11,6){\circle*{0.3}}
\Put(12,6){\circle*{0.3}}
\xLINE(12,6)(12,5)
\Put(12,5){\circle*{0.3}}
\Put(11.9,4){$\vdots$}
\xLINE(12,3.5)(12,2.5)
\Put(12,3.5){\circle*{0.3}}
\Put(12,2.5){\circle*{0.3}}
\Put(11.8,1){$\Gamma_g$}

\put(12.3,2.2){$\ell_g$}

\Put(11.7,6.3){\footnotesize $st_g$}

\Put(9,1){$\;\;\cdots\;\;$}


\xLINE(12,6)(13,6)
\Put(13,6){\circle*{0.3}}
\Put(13,6){$\;\;\ldots\;\;$}
\Put(15,6){\circle*{0.3}}
\xLINE(15,6)(16,6)
\Put(16,6){\circle*{0.3}}

\xVECTOR(16,6)(17,7)

\Put(12.7,5.5){$\underbrace{\;\;\;\;\;\;\;\;\;\;\;\;\;\;\;}$}
\Put(13.6,4.2){{\footnotesize Tail}}

\end{Picture}
\end{center}
  \caption{Dual graph of a divisorial valuation}
  \label{fig2}
\end{figure}
In Figure \ref{fig2} we have depicted the dual graph  associated with the configuration ${\mathcal C}_{\nu}$, when $\nu$ is a divisorial valuation (without the above mentioned labels). Here, we add some more labels: the vertex labeled with $\mathbf{1}$ corresponds to the exceptional divisor of the first blowing-up. Vertices different from $\mathbf{1}$ which are adjacent to a unique vertex are called \emph{dead ends} (labeled as $\ell_1, \ell_2 \ldots, \ell_g$ in Figure \ref{fig2}), and those adjacent to three vertices are called \emph{star vertices}. We have labeled the star vertices with $st_1, st_2, \ldots, st_g$ (where the index refers to the order of appearance in the sequence of blowing-ups).
Consider also the following ordering on the set of vertices: $\alpha \preccurlyeq \beta$ iff the path in the dual graph joining $\mathbf{1}$ and $\beta$ goes through $\alpha$.
For each $i=1, 2, \ldots,g$, we denote by $\Gamma_i$  the subgraph given by the vertices $\alpha$ such that $st_{i-1} \preccurlyeq \alpha \preccurlyeq \ell_i$  (where $st_0:=\mathbf{1}$) and the edges joining them. The vertices of each subgraph $\Gamma_i$ corresponding to free divisors are some of the first (consecutive) ones and $\ell_i$ (we call them the \emph{free part} of $\Gamma_i$); moreover, they correspond with the coefficients $a_{s_{i-1} \beta}$ of the $i$th free row of the Hamburger-Noether expansion. The \emph{free rows} are the first one, those rows whose left-hand side term is $\omega_{s_j-1}$ for $j=1, 2, \ldots,g-1$, and the last row (\ref{DVI}) whenever there exists a coefficient $a_{s_g k_g}$.

If the divisor defining $\nu$ is free, then there is a finite sequence of vertices corresponding to free divisors which appear after $st_g$ (the \emph{tail}, in Figure \ref{fig2}). Otherwise this tail does not appear (this case corresponds to the fact that (\ref{DVI}) is not a free row).

Setting $t:=w_{s_g}$ and $z:=w_{s_g+1}$, by backward substitution in the Hamburger-Noether expansion of $\nu$, we obtain parametric equations $u=u(t,z), v=v(t,z)$ in $k[\![t,z]\!]$ such that, if $h\in R$, then $\nu(h)={\rm ord}_t\;h(u(t,z),v(t,z))$ \cite[Section 3]{galindo}. \emph{Notice that the series $u(t,z)$ and $v(t,z)$ have a finite number of non-zero terms}.
\\


{\it $\diamond$ Irrational valuations} have a Hamburger-Noether expansion whose last part consists of a
free row as (\ref{DVI}), with $\nu(w_{s_{g}+1}) \not =0$, followed by infinitely many rows of the form
\begin{equation}\label{prr}
w_{i- 1} = w^{h_{i}}_{i}w_{i+1} \ \ \ \mbox{with } i > s_g.
\end{equation}
This means that the the configuration ${\mathcal C}_{\nu}=\{p_i\}_{i=1}^{\infty}$ is infinite and satisfies the following condition: there exists an index $i_0$ such that, for all $j\geq i_0$, all the points $p_j$ are satellite and they are not proximate to the same point. The rows (\ref{prr}) of the Hamburger-Noether expansion describe the last infinitely many satellite points. The dual graph of ${\mathcal C}_{\nu}$ (with infinitely many vertices) is shown in Figure \ref{fig3} (see also \cite[Section 9]{ja20}).

Let $k\langle t \rangle$ be the ring of formal series $\sum_{r\in \mathbb{R}} a_r t^r, a_r\in k$, such that the set $\{r\in \mathbb{R}\mid a_r\neq 0\}$ is well-ordered. Consider the non-rational number defined by the infinite continued fraction
$$\gamma:=h_{s_g+1}+\frac{1}{h_{s_g+2}+\cdots}.$$
If we write $w_{s_g+1}=t$ and $w_{s_g}=t^\gamma$ in the Hamburger-Noether expansion of $\nu$ and we perform backward substitution, we obtain parametric equations $u=u(t), v=v(t)\in k\langle t \rangle$ such that $\nu(h)={\rm ord}_t\;h(u(t),v(t))$ for all $h\in R$ \cite[Section 6]{galindo}. \emph{Notice that the series $u(t)$ and $v(t)$ have finitely many non-zero terms}.

\begin{figure}[h]
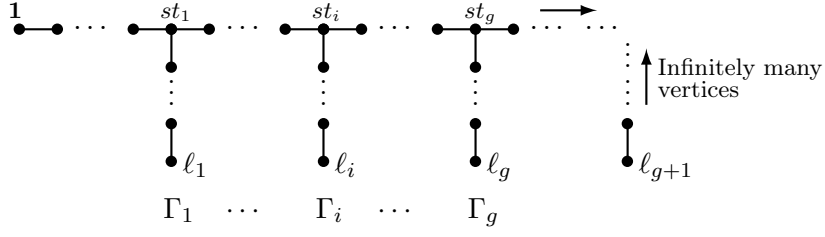


\begin{center}
\setlength{\unitlength}{0.5cm}%
\begin{Picture}(0,0)(20,8)
\thicklines


\xLINE(0,6)(1,6)
\Put(0,6){\circle*{0.3}}
\Put(1,6){\circle*{0.3}}
\put(1,6){$\;\;\ldots\;\;$}
\xLINE(3,6)(4,6)
\Put(3,6){\circle*{0.3}}
\Put(4,6){\circle*{0.3}}
\xLINE(4,6)(4,5)
\Put(4,5){\circle*{0.3}}
\Put(3.9,4){$\vdots$}
\xLINE(4,3.5)(4,2.5)
\Put(4,3.5){\circle*{0.3}}
\Put(4,2.5){\circle*{0.3}}
\Put(3.8,1){$\Gamma_1$}

\put(4.3,2.2){$\ell_1$}

\Put(3.7,6.3){\footnotesize $st_1$}
\Put(-0.3,6.3){\footnotesize $\mathbf{1}$}



\xLINE(4,6)(5,6)
\Put(4,6){\circle*{0.3}}
\Put(5,6){\circle*{0.3}}
\Put(5,6){$\;\;\ldots\;\;$}
\xLINE(7,6)(8,6)
\Put(7,6){\circle*{0.3}}
\Put(8,6){\circle*{0.3}}
\xLINE(8,6)(8,5)
\Put(8,5){\circle*{0.3}}
\Put(7.9,4){$\vdots$}
\xLINE(8,3.5)(8,2.5)
\Put(8,3.5){\circle*{0.3}}
\Put(8,2.5){\circle*{0.3}}
\Put(7.8,1){$\Gamma_i$}

\put(8.3,2.2){$\ell_i$}

\Put(7.7,6.3){\footnotesize $st_i$}

\Put(5,1){$\;\;\cdots\;\;$}


\xLINE(8,6)(9,6)
\Put(8,6){\circle*{0.3}}
\Put(9,6){\circle*{0.3}}
\put(9,6){$\;\;\ldots\;\;$}
\xLINE(11,6)(12,6)
\Put(11,6){\circle*{0.3}}
\Put(12,6){\circle*{0.3}}
\xLINE(12,6)(12,5)
\Put(12,5){\circle*{0.3}}
\Put(11.9,4){$\vdots$}
\xLINE(12,3.5)(12,2.5)
\Put(12,3.5){\circle*{0.3}}
\Put(12,2.5){\circle*{0.3}}
\Put(11.8,1){$\Gamma_g$}

\put(12.3,2.2){$\ell_g$}

\Put(11.7,6.3){\footnotesize $st_g$}

\Put(9,1){$\;\;\cdots\;\;$}


\xLINE(12,6)(13,6)
\Put(13,6){\circle*{0.3}}
\Put(13,6){$\;\;\ldots\;\;\ldots$}
\Put(15.9,5){$\vdots$}
\Put(15.9,4){$\vdots$}
\Put(16,3.5){\circle*{0.3}}
\xLINE(16,3.5)(16,2.5)
\Put(16,2.5){\circle*{0.3}}

\xVECTOR(16.5,4)(16.5,5.5)

\xVECTOR(13.7,6.5)(15.2,6.5)

\Put(16.8,4.8){\footnotesize Infinitely many}
\Put(16.8,4.2){\footnotesize vertices}

\put(16.3,2.2){$\ell_{g+1}$}

\end{Picture}
\end{center}
  \caption{Dual graph of an irrational valuation}
  \label{fig3}
\end{figure}

A third type of plane valuations which will be a useful tool for our purposes are the so-called \emph{curve valuations}: these plane valuations have Hamburger-Noether expansion with the shape showed in Figure \ref{hne}, but with a last row of the form
\begin{equation}\label{uiui}
w_{s_{g}- 1}  = a_{s_{g}k_{g}}w^{k_{g}}_{s_{g}}+ \cdots.
\end{equation}
Notice that the values $a_{s_{g}i}$, for $i > k_g$, may vanish, but $a_{s_{g}k_g} \neq 0$. The configuration ${\mathcal C}_{\nu}$ is also infinite in this case: there exists an index $i_0$ such that $p_j$ is free for all $j\geq i_0$. The coefficients in the row given in (\ref{uiui}) of the Hamburger-Noether expansion determine this last infinite sequence of free points. In Figure \ref{fig4}, we have depicted the dual graph of a curve valuation.

\begin{figure}[h]
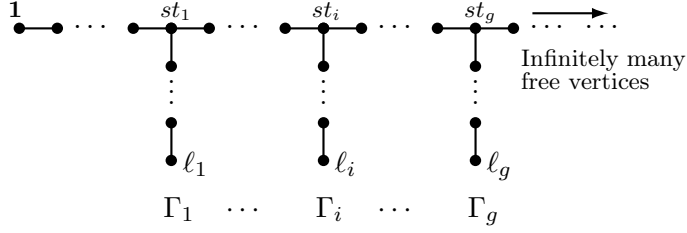


\begin{center}
\setlength{\unitlength}{0.5cm}%
\begin{Picture}(0,0)(20,8)
\thicklines


\xLINE(0,6)(1,6)
\Put(0,6){\circle*{0.3}}
\Put(1,6){\circle*{0.3}}
\put(1,6){$\;\;\ldots\;\;$}
\xLINE(3,6)(4,6)
\Put(3,6){\circle*{0.3}}
\Put(4,6){\circle*{0.3}}
\xLINE(4,6)(4,5)
\Put(4,5){\circle*{0.3}}
\Put(3.9,4){$\vdots$}
\xLINE(4,3.5)(4,2.5)
\Put(4,3.5){\circle*{0.3}}
\Put(4,2.5){\circle*{0.3}}
\Put(3.8,1){$\Gamma_1$}

\put(4.3,2.2){$\ell_1$}

\Put(3.7,6.3){\footnotesize $st_1$}
\Put(-0.3,6.3){\footnotesize $\mathbf{1}$}



\xLINE(4,6)(5,6)
\Put(4,6){\circle*{0.3}}
\Put(5,6){\circle*{0.3}}
\Put(5,6){$\;\;\ldots\;\;$}
\xLINE(7,6)(8,6)
\Put(7,6){\circle*{0.3}}
\Put(8,6){\circle*{0.3}}
\xLINE(8,6)(8,5)
\Put(8,5){\circle*{0.3}}
\Put(7.9,4){$\vdots$}
\xLINE(8,3.5)(8,2.5)
\Put(8,3.5){\circle*{0.3}}
\Put(8,2.5){\circle*{0.3}}
\Put(7.8,1){$\Gamma_i$}

\put(8.3,2.2){$\ell_i$}

\Put(7.7,6.3){\footnotesize $st_i$}

\Put(5,1){$\;\;\cdots\;\;$}


\xLINE(8,6)(9,6)
\Put(8,6){\circle*{0.3}}
\Put(9,6){\circle*{0.3}}
\put(9,6){$\;\;\ldots\;\;$}
\xLINE(11,6)(12,6)
\Put(11,6){\circle*{0.3}}
\Put(12,6){\circle*{0.3}}
\xLINE(12,6)(12,5)
\Put(12,5){\circle*{0.3}}
\Put(11.9,4){$\vdots$}
\xLINE(12,3.5)(12,2.5)
\Put(12,3.5){\circle*{0.3}}
\Put(12,2.5){\circle*{0.3}}
\Put(11.8,1){$\Gamma_g$}

\put(12.3,2.2){$\ell_g$}

\Put(11.7,6.3){\footnotesize $st_g$}

\Put(9,1){$\;\;\cdots\;\;$}


\xLINE(12,6)(13,6)
\Put(13,6){\circle*{0.3}}
\Put(13,6){$\;\;\ldots\;\;\ldots$}

\xVECTOR(13.5,6.4)(15.5,6.4)

\Put(13.2,5){\footnotesize Infinitely many}
\Put(13.2,4.4){\footnotesize free vertices}

\end{Picture}
\end{center}
  \caption{Dual graph of a curve valuation}
  \label{fig4}
\end{figure}

\medskip

\medskip

As mentioned, Hamburger-Noether expansions were first introduced for irreducible germs of plane curves. Indeed, the Hamburger-Noether expansion of a germ $C_f$ defined by $f(u,v)=0$ has an expression as in Figure \ref{hneG}, where
$\bar{u} = u + (f) \in R/(f)$ and $\bar{v} = v + (f) \in R/(f)$ (see \cite{campillo}).
\begin{figure}[h]
\[
\begin{array}{lccl}
&\bar{v} & = & a_{01}\bar{u}+a_{02}\bar{u}^{2}+ \cdots + a_{0h_{0}}\bar{u}^{h_{0}}+\bar{u}^{h_{0}}\bar{w}_{1} \\
&\bar{u} & = & \bar{w}^{h_{1}}_{1}\bar{w}_{2} \\
&\vdots & \nonumber & \vdots \\
&\bar{w}_{s_{1}- 2} & = & \bar{w}^{h_{s_{1}-1}}_{s_{1}-1}\bar{w}_{s_{1}}\\
&\bar{w}_{s_{1}-1} & = & a_{s_{1}k_{1}}\bar{w}^{k_{1}}_{s_{1}} + \cdots
 +a_{s_{1}h_{s_{1}}}\bar{w}^{h_{s_{1}}}_{s_{1}}+\bar{w}^{h_{s_{1}}}_{s_{1}}w_{s_{1}+1}
 \\
& \vdots & \nonumber & \vdots \\
&\bar{w}_{s_{g}- 1} & = & a_{s_{g}k_{g}}\bar{w}^{k_{g}}_{s_{g}} + \cdots.
\end{array}
\]
\caption{Hamburger-Noether expansion of a germ}
\label{hneG}
\end{figure}

We conclude by adding that, when $\nu$ is divisorial, choosing suitable coordinates and germs: $q_0$ given by $u=0$ and $q_i$, $1 \leq i \leq g+1$, whose Hamburger-Noether expansion is as in Figure \ref{hneG} but with last row $$
\bar{w}_{s_{i-1}- 1} = a_{s_{i-1}k_{i-1}}\bar{w}^{k_{i-1}}_{s_{i-1}} + \cdots
+a_{s_{i-1}h_{s_{i-1}}}\bar{w}^{h_{s_{i-1}}}_{s_{i-1}} + \cdots,
$$
it happens that $\{q_i\}_{i=0}^{g+1}$ is a generating sequence of the valuation $\nu$,  with $q_{g+1}$ a general element of $\nu$ and $\nu(q_i)=\betabarra_{i}$, cf.~\cite{ja20}.

\subsection{Further invariants of plane valuations}\label{formulas}
Plane valuations admit some numerical invariants that help to study them. Since we are only interested in divisorial and irrational valuations, we will briefly recall these invariants for the mentioned valuations. We preserve notations as above.

First we consider the {\it sequence of values}. For each $p_i\in {\mathcal C}_{\nu}$, set $m_i:=\min \{\nu(\varphi)\mid \varphi\in \mathfrak{m}_i\setminus \{0\}\}$, where $\mathfrak{m}_i$ is the maximal ideal of the ring $R_i={\mathcal O}_{X_{i-1},p_i}$. The sequence $\{m_i\}_{i\geq 1}$ is called the {\it sequence of values} of $\nu$, and by construction of the Hamburger-Noether expansion of $\nu$, it can be obtained from the sequence $\{\nu(\omega_i)\}_{i\geq 0}$ by repeating $h_i$ times each value $\nu(\omega_i)$ (with $\omega_0:=u$) \cite[1.5.1]{d-g-n}.

We consider also the {\it Puiseux exponents:} they are real numbers $\beta'_{0},\beta'_{1},\ldots , \beta'_{g+1}$, defined by $\beta'_0 := 1$ and for $0 \leq j \leq g$,
\[
\beta'_{j+1} := h_{s_j} - k_j +1 + \frac{1}{h_{s_j +1} +
\frac{1}{\ddots}},
\]
where the integers $s_j$, $h_{s_j}$, $h_{s_j +1}$ and $k_j$ can be read off from the Hamburger-Noether expansion of $\nu$.
The Puiseux exponents are in fact rational numbers except for $\beta'_{g+1}$ in the case $\nu$ is irrational.

Hence we can write $\beta'_j = p_j/n_j$ with $\gcd(p_j,n_j)=1$, $e_j =\nu (w_{s_j})$ for  $0 \leq j \leq g$ ($w_{s_0}:=w_0=u$), and $r_i=\nu(w_i)$ for $i \geq 0$, and define the {\it
characteristic sequence} $\{\beta_j\}_{j=0}^{g+1}$ of  $\nu$ as
\[
\beta_0:=e_0,\;\;\; \beta_{j+1} := \beta_j + (h_{s_j} - k_j) e_j + r_{s_j +1},
\]
as well as the sequence $\{\bar{\beta}_j\}_{j=0}^{g+1}$ of {\it maximal
contact values} of $\nu$ as
\[
\bar{\beta}_0:=e_0,\;\;\; \bar{\beta}_{j+1} := n_j \bar{\beta}_j + (h_{s_j} - k_j) e_j +
r_{s_j +1}.
\]
It is worth mentioning that the value semigroup $S$ of $\nu$ (associated to $R$) is generated by the set of maximal
contact values of $\nu$.

From the previous formulae for Puiseux exponents and maximal contact values, one can deduce that $n_j=e_{j-1}/e_j$, where $e_j= \gcd(\bar{\beta}_0, \bar{\beta}_1, \ldots, \bar{\beta}_j)$, and
\begin{equation}
\label{Delta}
\beta'_{j+1} = \frac{\bar{\beta}_{j+1}- n_j \bar{\beta}_j}{e_j}+1, \ \ \ \mbox{for} \ \ 0 \leq j \leq g.
\end{equation}

All the previous formulae involving Puiseux exponents and maximal contact values are also valid for \emph{curve valuations}, but in this case $j$ can only take values strictly less than $g$ (i.e., only $\bar{\beta}_0, \bar{\beta}_1, \ldots, \bar{\beta}_g$ and $\beta'_0, \beta'_1,\ldots, \beta'_g$ are defined).

\medskip



The dual graph  of a valuation $\nu$ is an equivalent datum to the structure of its Hamburger-Noether expansion (that is, all the rows except the specific choice of the coefficients $a_{\alpha \beta}$), cf.~\cite{d-g-n}. Both of them  determine, and are determined by, the Puiseux exponents of $\nu$. In fact, for each $i\in \{1, 2, \ldots,g\}$, the continued fraction expression of the rational number $\beta'_{i}$ determines the subgraph $\Gamma_i$ (see Figures \ref{fig2}, \ref{fig3} and \ref{fig4}) and, if $\nu$ is divisorial (respectively, irrational), $\beta'_{g+1}$ determines the \emph{tail} (respectively, the infinite subgraph $\Gamma\setminus \cup_{i=1}^g \Gamma_i$); see \cite{d-g-n} for more details.

Each of the data --- the sequence of values, Puiseux exponents, maximal contact values and dual graph --- can be obtained from any of the others \cite[Theorem 1.11]{d-g-n}.

\subsection{Irrational valuations as limits of divisorial valuations}

Unlike the case of divisorial valuations, in which we distinguish between \emph{non-normalized} ($\nu$) and \emph{normalized} ($\nu^{N}$) equivalent valuations, all the irrational valuations $\nu$ that we will consider in this paper will be assumed to be \emph{normalized} in the sense that $\nu(\mathfrak{m})=1$; in this case, we use $\nu$ or $\nu^N$ interchangeably.

The following result is a straightforward consequence of the proof of Theorem 6.1 in \cite{galindo}:

\begin{pro}\label{limit}
Let $\nu$ be an irrational valuation and, for all $i\geq 1$, let $\nu_i$ be the divisorial valuation associated to the divisor $E_i$ defined by $\pi_i$ in the sequence of blowing-ups (\ref{uno}) corresponding to $\nu$. Then
$$
\nu(f)=\lim_{i \rightarrow \infty} \nu_i^N(f),\ \ \mbox{for~all} \ \  f \in K.
$$
\end{pro}

\subsection{Volume and normalized volume of a valuation}
\label{volumen}



Let $\nu$ be a divisorial valuation. Set $\mathcal{C}_\nu := \{p_i\}_{i=1}^{s}$ for the corresponding configuration of infinitely near points. The sequence of values $\{m_i\}_{i=1}^s$  determines $\nu$ because of the equality
\begin{equation}
\label{AAA}
\nu(\psi)=\sum_{i=1}^s m_i\cdot {\rm mult}_{p_i}(\psi), \ \ \mbox{for~every} \ \psi \in R.
\end{equation}

According to
\cite{e-l-s}, the \emph{volume} of $\nu$ is defined as
\[
{\rm vol}(\nu):=\lim_{\alpha \rightarrow \infty} \frac{\mathrm{length} (R/{\mathcal P}_{\alpha})}{\alpha^2/2}.
\]

Write $\alpha=r \sum_{i=1}^s {m_i^2}$ for some $r \in \mathbb{N}$, then $\mathrm{length} (R/{\mathcal P_{\alpha}})=\sum_{i=1}^s \frac{r m_i(r m_i+1)}{2}$ (cf.
\cite[4.7]{casas}), and therefore
\[
{\rm vol}(\nu)=\left(\sum_{i=1}^s m_i^2\right)^{-1}= \frac{1}{\nu(q_{g+1})} =\frac{1}{\betabarra_{g+1}}.
\]

We also define the \emph{normalized volume} of $\nu$, ${\rm vol}^N(\nu)$,  as the volume of the normalized valuation $\nu^N$, that is:
\[
{\rm vol}^N(\nu):={\rm vol}\left(\nu^N \right)=\frac{\betabarra_0^2}{\bar{\beta}_{g+1}}.
\]

Proposition \ref{limit} allows us to define the volume (or normalized volume) of an irrational valuation $\nu$ as
${\rm vol}(\nu)={\rm vol}^N(\nu):=\lim_{i\rightarrow \infty} {\rm vol}^N(\nu_i)$,
where $\{\nu_i\}_{i\geq 1}$ is the sequence of divisorial valuations defined by the exceptional divisors appearing in the sequence (\ref{uno}).

\subsection{Minimal valuations}\label{minimal}
Let $\mathbb{P}^2 := \mathbb{P}^2_k$ be the projective plane over an algebraically closed field $k$ and fix projective coordinates $X, Y, Z$. For the sake of simplicity, assume that $p=p_1=(1:0:0)\in \mathbb{P}^2$ and consider local coordinates $(u,v)$, $u=Y/X$ and $v=Z/X$, around $p$. Let $\nu$ be a divisorial or irrational valuation of the fraction field of the local ring $R:= {\mathcal O}_{{\mathbb P}^2,p}$, centered at $R$.


For each positive integer $d$ we denote
$$\mu_d(\nu):=\max\{\nu(f)\mid f\in k[u,v] \mbox{ and } \deg(f)\leq d\},$$
$$\mu^N_d(\nu):=\mu_d(\nu^N)=\max\{\nu^N(f)\mid f\in k[u,v] \mbox{ and } \deg(f)\leq d\},$$
as well as
$$\hat{\mu}(\nu):=\lim_{d\rightarrow \infty} \frac{\mu_d(\nu)}{d},\;\;\;
\hat{\mu}^N(\nu):=\hat{\mu}\left( \nu^N \right)=\lim_{d\rightarrow \infty} \frac{\mu^N_d(\nu)}{d}.$$

As stated in \cite{d-h-k-r-s}), the following inequality holds:
$$\hat{\mu}(\nu)\geq \sqrt{[{\rm vol}(\nu)]^{-1}};\;\;\; \mbox{or, equivalently, }\hat{\mu}^N(\nu)\geq \sqrt{[{\rm vol}^N(\nu)]^{-1}}\;\;$$
and, accordingly, we present the following definition.
\begin{de}
{\rm
A valuation $\nu$ as above is said to be \emph{minimal} if
$$\hat{\mu}(\nu)=\sqrt{\frac{1}{{\rm vol}(\nu)}}\; ; \ \mbox{or, equivalently, if~} \ \hat{\mu}^N(\nu)=\sqrt{\frac{1}{{\rm vol}^N(\nu)}}.
$$}
\end{de}

\section{The metric spaces of valuations $V_{\delta}$ and continuity of $\hat{\mu}^N$}\label{metric}

From now on, let us fix a  curve (plane) valuation $\delta$ with Hamburger-Noether expansion $H$ (with respect to a fixed system of parameters $(u,v)$) as explained in Subsection \ref{s:hne}. We define $V_{\delta}$ as the set of \emph{divisorial} or \emph{irrational} plane valuations $\nu$ satisfying the following conditions:
\begin{itemize}
\item[$\diamond$] The Hamburger-Noether expansions of $\nu$ and $\delta$ coincide \emph{up to the row} where $w_{s_g-2}$ is the left-hand side of the equality.
\item[$\diamond$] The row with left-hand side equal to $w_{s_g-1}$ has the form given in (\ref{DVI}) and, either all the coefficients $a_{s_g \beta}$ are zero and $k_g=h_{s_g}$ (in this case $\nu$ is divisorial and this is the last row), or this is the last \emph{free} row and all the (finitely many) coefficients $a_{s_g \beta}$
coincide exactly with the corresponding coefficients of the last row of the Hamburger-Noether expansion of $\delta$.
\end{itemize}
In other words, $V_{\delta}$ consists of those divisorial and irrational valuations whose dual graph $\Gamma$ (see Figures \ref{fig2} and \ref{fig3}) contains $g$ subgraphs $\Gamma_1, \Gamma_2, \ldots, \Gamma_g$, the subgraph given by $\cup_{i=1}^g \Gamma_i$ coincides with the one of $\delta$, and the infinitely near points corresponding to vertices in $\cup_{i=1}^g \Gamma_i$ and the \emph{free part} of  $\Gamma\setminus \cup_{i=1}^g \Gamma_i$ also coincide with those of $\delta$.



It is clear that $\nu \in V_\delta$ is determined from both $H$ and the  Puiseux exponent $\beta'_{g+1}$ of $\nu$. Notice that $\nu$ has $g+2$ Puiseux exponents, $\{\beta'_j\}_{j=0}^{g+1}$,
when either $\nu$ is divisorial and the divisor defining $\nu$ is either free or corresponds to the vertex $st_g$ (see Figure \ref{fig2}), or $\nu$ is irrational. Otherwise, $\nu$ has $g+3$ Puiseux exponents, $\{\beta'_j\}_{j=0}^{g+2}$ with $\beta'_{g+2}=1$.

As a consequence, the {\it normalized volume} of a divisorial valuation $\nu$ in $V_\delta$ is either
\begin{equation}\label{tt}
\mathrm{vol}^N(\nu) = \bar{\beta}_0^2/\bar{\beta}_{g+1} \ \ \mbox{or}  \ \ \mathrm{vol}^N(\nu) = \bar{\beta}_0^2/\bar{\beta}_{g+2}
\end{equation}
depending on which of the above two cases we are in.

When considering an irrational valuation $\nu\in V_{\delta}$, let $\{\nu_i\}_{i=1}^\infty$ be the sequence of divisorial valuations associated to the exceptional divisors of the blowing-ups $\pi_i$ appearing in the sequence (\ref{uno}). It is clear that there exists an index $i_0$ such that $\nu_i$ belongs to $V_{\delta}$ for all $i\geq i_0$, and has $g+3$ associated Puiseux exponents; therefore
$${\rm vol}(\nu)={\rm vol}^N(\nu)=\lim_{i\rightarrow \infty} {\rm vol}^N(\nu_i)= \lim_{i\geq i_0} \frac{(\betabarra_0^i)^2}{\betabarra_{g+2}^i}=\lim_{i\geq i_0} \frac{\betabarra_0^i}{e_g^i} \frac{\betabarra_0^i}{\betabarra_{g+1}^i},$$
where the superscript $i$ stands for those values associated with the valuation $\nu_i$, and the equality $\betabarra_{g+2}^i=e_g^i\betabarra_{g+1}$ comes from Equality (\ref{Delta}) and the fact that $\beta_{g+2}'^{i}=1$, because $\nu_i$ is defined by a non-free divisor for all $i\geq i_0$.  Notice that
$$\lim_{i\geq i_0} \frac{\betabarra_0^i}{\betabarra_{g+1}^i}=\frac{1}{\betabarra_{g+1}},$$
where $\betabarra_{g+1}$ is the ($g+2$)-th contact maximal value associated to $\nu$.

Moreover, denoting by  $\betabarra_0^{\delta}$ the first maximal contact value of the curve valuation $\delta$, the quotient $\betabarra_0^i/e_g^i$ equals $\betabarra_0^{\delta}$ for all $i\geq i_0$, and we obtain that
\[
{\rm vol}(\nu)={\rm vol}^N(\nu)=\frac{\betabarra_0^{\delta}}{\betabarra_{g+1}}.
\]


Inverses of the normalized volumes of valuations $\nu$ in $V_{\delta}$ are related with the Puiseux exponents $\beta'_{g+1}$ by means of an affine function, as the following result shows:

\begin{lem}
\label{lemauno}
There exist positive rational numbers $A$ and $B$, depending only on the dual graph of $\delta$, such that
$[{\rm vol}^N(\nu)]^{-1}=A(\beta'_{g+1}-1)+B$ for any valuation $\nu\in V_{\delta}$,
where $\beta'_{g+1}$ is the corresponding Puiseux exponent of $\nu$.
\end{lem}
\begin{proof}
Firstly, assume that $\nu$ is divisorial and is defined by either the divisor associated with the vertex $st_g$ or a free divisor. Equality (\ref{Delta}) implies that
\[
\bar{\beta}_{g+1} = (\beta'_{g+1} -1) e_g + n_g \bar{\beta}_{g},
\]
and by (\ref{tt}) it holds that
$$
\frac{1}{\mathrm{vol}^N(\nu)} = \frac{\bar{\beta}_{g+1}}{(\bar{\beta}_{0})^2} = \frac{1}{(\bar{\beta}_{0})^2} (\beta'_{g+1} -1) e_g + \frac{n_g \bar{\beta}_{g}}{(\bar{\beta}_{0})^2}.
$$
Note that here $e_g=1$, $n_g=e_{g-1}$, $\betabarra_0=\betabarra_0^{\delta}$, $e_{g-1}/\betabarra_0=e^{\delta}_{g-1}/\betabarra^{\delta}_0$ and $\betabarra_g/\betabarra_0=\betabarra^{\delta}_g/\betabarra^{\delta}_0$, $e_{g-1}^{\delta}$ being the greatest common divisor of the first $g$ maximal contact values of $\delta$. Hence
\begin{equation}
\label{mm}
\frac{1}{\mathrm{vol}^N(\nu)} = \frac{1}{(\bar{\beta}^{\delta}_{0})^2} (\beta'_{g+1} -1) + \frac{e^{\delta}_{g-1} \bar{\beta}^{\delta}_{g}}{(\bar{\beta}^{\delta}_{0})^2}.
\end{equation}

When $\nu$ is divisorial but defined by a non-free divisor which is not the one associated to $st_g$, it holds that $\bar{\beta}_{g+2} =  n_{g+1} \bar{\beta}_{g+1}$ and so
\[
\bar{\beta}_{g+2} = e_g^2 (\beta'_{g+1} -1)  + n_{g+1} n_g  \bar{\beta}_{g},
\]
which gives
$$
\frac{1}{\mathrm{vol}^N(\nu)}
= \frac{\bar{\beta}_{g+2}}{(\bar{\beta}_{0})^2} = \left(\frac{e_g}{\bar{\beta}_{0}}\right)^2 (\beta'_{g+1} -1)  + \frac{e_{g-1}}{\bar{\beta}_{0}} \frac{ \bar{\beta}_{g}}{\bar{\beta}_{0}}.
$$
Taking into account that $\betabarra_0/e_g=\betabarra_0^{\delta}$, we obtain Formula (\ref{mm}) in this case too.

If $\nu$ is irrational and $\{\nu_i\}_{i\geq 1}$ is the sequence of divisorial valuations associated with the exceptional divisors of the blowing-ups $\pi_i$ in (\ref{uno}), we have that
$$\left[{\rm vol}^N(\nu)\right]^{-1}=\lim_{i\rightarrow \infty} \left[{\rm vol}(\nu_i)\right]^{-1}\;\;\mbox{ and }\;\; \beta'_{g+1}=\lim_{i\rightarrow \infty} \beta'^i_{g+1}.$$
This extends Formula (\ref{mm}) also to irrational valuations.

\end{proof}

Observe that $1$ is the minimum value for the Puiseux exponent $\beta'_{g+1}$ of a valuation in $V_{\delta}$ (attained for the divisorial valuation defined by the divisor associated with the vertex $st_g$). Therefore, by the proof of Lemma \ref{lemauno}, the set of inverses of normalized volumes, $\left[{\rm vol}^N(\nu)\right]^{-1}$, when $\nu$ varies in $V_{\delta}$, runs over the interval
$$\Delta_{\delta}:=\left[\frac{e^{\delta}_{g-1} \bar{\beta}^{\delta}_{g}}{(\bar{\beta}^{\delta}_{0})^2},+\infty\right).$$

Any valuation in $V_{\delta}$ is determined by the Puiseux exponent $\beta'_{g+1}$. Indeed, the infinitely near points associated with the vertices of the subgraph $\cup_{i=1}^g \Gamma_i$ and the free part of the graph $\Gamma\setminus \cup_{i=1}^g \Gamma_i$ are determined by $\delta$; since $\beta'_{g+1}$ allows us to recover the subgraph $\Gamma\setminus \cup_{i=1}^g \Gamma_i$, the remaining satellite points are determined too. This fact together with Lemma \ref{lemauno} proves the existence of a bijection $\phi: \Delta_{\delta}\rightarrow V_{\delta}$ that assigns, to each $t\in \Delta_{\delta}$, the unique valuation $\nu_t \in V_{\delta}$ such that $\left[{\rm vol}^N(\nu_t)\right]^{-1}=t$ (see \cite{fj}, where the inverse of the normalized volume of a valuation is named skewness). Hence, we can write
$$V_{\delta}=\{\nu_t\}_{t\in \Delta_{\delta}}.$$
Moreover, we endow $V_{\delta}$ with a structure of metric space induced by $\phi$ and the absolute value in $\Delta_{\delta}$.


\begin{pro}\label{proposiciondos}
With the above notation and under the assumptions in Section \ref{minimal}, for any fixed element $ f \in R:= {\mathcal O}_{{\mathbb P}^2,p}$, the map $\eta_f: V_{\delta}\rightarrow  \mathbb{R}$, $\eta_f (\nu_t) = \nu_t^N(f)$  is Lipschitz continuous. As a consequence, the map  $V_{\delta}\rightarrow  \mathbb{R}$ defined by $\nu_t\mapsto \hat{\mu}^N(\nu_t)$ is also Lipschitz continuous.
\end{pro}

\begin{proof}
We have to prove that, for every pair
of indices $t_1,t_2\in \Delta_{\delta}$, there
exists a real constant $D$, depending only on $H$ and $f$, such that
\[
\left|\nu_{t_1}^N (f) - \nu_{t_2}^N (f) \right| \leq D \left|t_1 -t_2\right|.
\]
Our second statement follows  easily from the fact that $\hat{\mu}^N (\nu_t)=\sup_{f\in k[u,v]}\{\nu_t^N (f)/\deg(f)\}$.

By the triangle inequality, we may assume that {\it $f$ is an analytically irreducible element}. In addition,
by Lemma \ref{lemauno}, denoting by $\beta^{'i}_{j}$ the $j$-th Puiseux exponent of  $\nu_{t_i}$,
$i=1,2$, we can replace $t_1 -t_2$ with $\beta^{'1}_{g+1} - \beta^{'2}_{g+1}$ when necessary.

Consider the embedded resolution of the germ of curve $C_f$ given by $f(u,v)=0$ and, when necessary,
successive blows-ups at the points where the germ meets the exceptional divisor. Denote by $\alpha_f^i$
that vertex in the dual graph of $\nu_{t_i}$ corresponding to the last created  divisor $E_{\alpha_f^i}$  that meets the strict transform of $C_f$; for $i=1,2$, the superscript $i$ always refers to the valuation $\nu_{t_i}$.

For a start, {\it assume that $\alpha_f^i \leq st_g^i$}. Clearly, the above inequality happens for both indexes $i$ because $ st_g^1 =  st_g^2 := st_g $. Then (\ref{AAA}), and Proposition \ref{limit} in the irrational case, prove that $\nu_{t_1}^N(f) = \nu_{t_2}^N(f)$ and the inequality in the statement holds.

From now on, assume that $st_g < \alpha_f^1$, $st_g <  \alpha_f^2$ and that the two valuations $\nu_{t_1}$ and $\nu_{t_2}$ are divisorial. The irrational case follows from the divisorial case and Proposition \ref{limit}, so let us restrict ourselves to the divisorial case. From \cite{delgado}, it can be deduced that, if $C_i$ denotes a suitable germ of curve in $R$ defined by a general element of the valuation $\nu_{t_i}$ and $(\cdot , \cdot)$  intersection multiplicity at $p$, the values $\nu_{t_i}(f)= (C_i,C_f)$ only admit the following three possibilities:
\begin{description}
  \item[a)] $\nu_{t_i}(f)=e_{g-1}^f \bar{\beta}_g^i + c_i^f e_g^f e^i_g$, where $c_i^f$ is the number of common free points after $st_g$ corresponding to the dual graphs of $C_i$ and $C_f$. This situation happens when the free points in the dual graphs of $C_i$ and $C_f$ do not coincide up to the last free point of one of them (which must have a satellite point after that last free point).
  \item[b)] $\nu_{t_i}(f)=e^f_g \bar{\beta}_{g+1}^i$, which holds when $\alpha^i_f \in [st^i_{g+1}, \ell^i_{g+1}]$, i.e., $st^i_{g+1} \preccurlyeq \alpha^i_f \preccurlyeq  \ell^i_{g+1}$.
  \item[c)] $\nu_{t_i}(f)=e^i_g \bar{\beta}_{g+1}^f$, which holds when $\alpha^i_f \not \in [st^i_{g+1}, \ell^i_{g+1}]$ and it does not correspond to a free divisor.
\end{description}
Let us show that the result follows for any pair coming from the above three possibilities and the proposition will be proved. We also notice that  the situation where one valuation (say $\nu_{t_1}$) corresponds to the case a) and the other one to the case c) is not possible. This is so because $\nu_{t_1}$ and $\nu_{t_2}$ either have the same free points or one of them adds new ones to the other.

Let us prove the inequality for the five remaining cases:

i) {\it If both valuations $\nu_{t_1}$ and $\nu_{t_2}$ are in case a),}  then
\[
\left|\nu_{t_1}^N (f) - \nu_{t_2}^N (f) \right| = e_{g-1}^f \left( \frac{\bar{\beta}_{g}^1}{\bar{\beta}_{0}^1} - \frac{\bar{\beta}_{g}^2}{\bar{\beta}_{0}^2}
 \right) +e_{g}^f
 \left( \frac{e_{g}^1}{\bar{\beta}_{0}^1} c_1^f - \frac{e_{g}^2}{\bar{\beta}_{0}^2} c_2^f
 \right)
\]
and since $\bar{\beta}_{g}^1 /\bar{\beta}_{0}^1 = \bar{\beta}_{g}^2 /\bar{\beta}_{0}^2$ and $ e_{g}^1 /\bar{\beta}_{0}^1 = e_{g}^2 /\bar{\beta}_{0}^2 = e_{g}^\delta /\bar{\beta}_{0}^\delta$, we get
\[
\left|\nu_{t_1}^N (f) - \nu_{t_2}^N (f)\right|  \leq \frac{e^f_ g e^\delta_g}{\bar{\beta}_0^\delta}
\left|\beta^{'1}_{g+1} - \beta^{'2}_{g+1}\right|,
\]
where $e^f_ g e^\delta_g/\bar{\beta}_0^\delta$ depends only on $H$ and $C_f$ and the inequality
holds from the expression of the values $\beta^{'1}_{g+1}$ and $\beta^{'2}_{g+1}$ as continued fractions.

ii) {\it If the valuations $\nu_{t_1}$ and $\nu_{t_2}$ are in cases a) and b) respectively,}  then
by (\ref{Delta}), we get the equality
\[
e_g^f \bar{\beta}_{g+1}^2 = e^f_g e^2_g (\beta^{'2}_{g+1} -1) + e^f_g n^2_g
\bar{\beta}_{g}^2.
\]
Taking into account that $e^f_{g-1} =e^f_g n^2_g$, we deduce that
\[
\left|\nu_{t_1}^N (f) - \nu_{t_2}^N (f)\right| = \left|\frac{\nu_{t_1}(f)}{\bar{\beta}_0^{1}} - \frac{\nu_{t_2}(f)}{\bar{\beta}_0^{2}}\right| =
\frac{e^f_g e^1_g}{\bar{\beta}_0^{1}} \left| c_1^f - (\beta^{'2}_{g+1}-1)  \right| \leq \frac{e^f_g
e^\delta_g}{\bar{\beta}_0^{\delta}} \left| \beta^{'1}_{g+1} - \beta^{'2}_{g+1} \right|,
\]
which concludes this case.

iii) {\it If the valuations $\nu_{t_1}$ and $\nu_{t_2}$ are both in case b),} then
\[
\left|\nu_{t_1}^N (f) - \nu_{t_2}^N (f)\right| =
\left|\frac{\nu_{t_1}(f)}{\bar{\beta}_0^{1}} - \frac{\nu_{t_2}(f)}{\bar{\beta}_0^{2}}\right| =
\left|\frac{e^f_g \bar{\beta}_{g+1}^{1}}{\bar{\beta}_0^{1}} - \frac{e^f_g
\bar{\beta}_{g+1}^{2}}{\bar{\beta}_0^{2}}\right|.
\]
 Multiplying by $e^1_g / \bar{\beta}_0^{1} = e^2_g / \bar{\beta}_0^{2} = e^\delta_g / \bar{\beta}_0^{\delta}$, we obtain that
\[
\frac{ e^\delta_g}{\bar{\beta}_0^{\delta}} \left|\frac{\nu_{t_1}(f)}{\bar{\beta}_0^{1}} - \frac{\nu_{t_2}(f)}{\bar{\beta}_0^{2}}\right| =
\frac{e^f_g e^\delta_g}{\bar{\beta}_0^{\delta}} \left| \frac{\bar{\beta}_{g+1}^{1}}{\bar{\beta}_0^{1}} -
\frac{ \bar{\beta}_{g+1}^{2}}{\bar{\beta}_0^{2}}\right|
\]
\[
= e^f_g \left| \frac{e^1_g \bar{\beta}_{g+1}^{1}}{(\bar{\beta}_0^{1})^2} - \frac{ e^2_g
\bar{\beta}_{g+1}^{2}}{(\bar{\beta}_0^{2})^2}\right| = e^f_g \left| \left[\mathrm{vol}^N(\nu_{t_1})\right]^{-1}
 - \left[\mathrm{vol}^N(\nu_{t_2})\right]^{-1}\right| = e^f_g \left|t_1 -t_2\right|.
\]

iv) {\it If the valuation $\nu_{t_1}$ corresponds to case b) and $\nu_{t_2}$ to case c),} then
\[
\left|\nu_{t_1}^N (f) - \nu_{t_2}^N (f)\right| = \left|\frac{\nu_{t_1}(f)}{\bar{\beta}_0^{1}} - \frac{\nu_{t_2}(f)}{\bar{\beta}_0^{2}}\right| =
\left|\frac{e^f_g \bar{\beta}_{g+1}^{1}}{\bar{\beta}_0^{1}} - \frac{e^2_g
\bar{\beta}_{g+1}^{f}}{\bar{\beta}_0^{2}}\right|.
\]
Now, on the one hand, it holds that
\[
\frac{e^f_g \bar{\beta}_{g+1}^{1}}{\bar{\beta}_0^{1}} =
\frac{e^1_g \bar{\beta}_{0}^{f}}{\bar{\beta}_0^{1}} \cdot \frac{
\bar{\beta}_{g+1}^{1}}{\bar{\beta}_0^{1}} = \bar{\beta}_{0}^{f} \left( \frac{e^1_g
\bar{\beta}_{g+1}^{1}}{(\bar{\beta}_0^{1})^2} \right).
\]
On the other hand, we get
\[
\frac{e^2_g \bar{\beta}_{g+1}^{f}}{\bar{\beta}_0^{2}} = \bar{\beta}_{0}^{f} \left( \frac{e^2_g
\bar{\beta}_{g+1}^{f}}{\bar{\beta}_0^{f} \bar{\beta}_0^{2}}      \right) = \bar{\beta}_{0}^{f} \left(
\frac{e^2_g e^f_g \bar{\beta}_{g+1}^{f}}{\bar{\beta}_0^{f} \bar{\beta}_0^{2} e^f_g}      \right).
\]
Formula (\ref{Delta}) implies that
\[
\frac{\bar{\beta}_{g+1}^{f}}{e^f_g} = \left(\beta^{'f}_{g+1} -1 \right) + \frac{ n_g^f
\bar{\beta}_{g}^{f}}{e^f_g} = \left( \beta^{'f}_{g+1} -1 \right) - \left( \beta^{'2}_{g+1} -1 \right) +
\frac{\bar{\beta}_{g+1}^{2}}{e_g^2},
\]
because $(n_g^f \bar{\beta}_{g}^{f})/e^f_g = (n_g^2 \bar{\beta}_{g}^{2})/e^2_g$. Thus
\[
\frac{e^2_g \bar{\beta}_{g+1}^{f}}{\bar{\beta}_0^{2}}= \bar{\beta}_{0}^{f} \left( \frac{e^2_g
e^f_g}{\bar{\beta}_0^{f} \bar{\beta}_0^{2}} \right) \left( \frac{ \bar{\beta}_{g+1}^{2}}{e^2_g} +
\eta \right),
\]
where $\eta = \beta^{'f}_{g+1} - \beta^{'2}_{g+1}$.  Since $e^f_g / \bar{\beta}_{0}^{f} = e^2_g/\bar{\beta}_{0}^{2}$, we get
\begin{align*}
\left|\nu_{t_1}^N (f) - \nu_{t_2}^N (f)\right|
& \leq \left |
\bar\beta_{0}^{f} \frac{e_g^1 \bar{\beta}_{g+1}^{1}}{(\bar\beta_{0}^{1})^2}
- \bar\beta_{0}^{f} \frac{e^2_g e^f_g \bar{\beta}_{g+1}^{2}}{\bar{\beta}_{0}^{f}\bar{\beta}_{0}^{2}
e^2_g} + \frac{e^2_g e_g^f}{\bar\beta_0^2} \eta  \right | \\
& \leq
\bar\beta_{0}^{f} \left| \frac{e^1_g \bar\beta_{g+1}^{1}}{(\bar{\beta}_{0}^{1})^2} - \frac{e^2_g
\bar\beta_{g+1}^{2}}{(\bar{\beta}_{0}^{2})^2}\right |  + \left | \frac{e^2_g e_g^f \bar\beta_0^f}{\bar\beta_0^2\bar\beta_0^f} \eta    \right | \\
& = \bar\beta_{0}^{f}\left| \left[\mathrm{vol}^N (
\nu_{t_1}) \right]^{-1} - \left[\mathrm{vol}^N ( \nu_{t_2}) \right]^{-1}\right| + \frac{(e^f_g)^2 \bar\beta_0^f}{(\bar\beta_0^f)^2} \left | \beta^{'f}_{g+1} - \beta^{'2}_{g+1} \right |.
\end{align*}
This concludes the proof in this case because it holds that $\beta^{'1}_{g+1} <\beta^{'f}_{g+1} <\beta^{'2}_{g+1}$ and, as a consequence,
\[
\left | \beta^{'f}_{g+1} - \beta^{'2}_{g+1} \right | \leq \left | \beta^{'1}_{g+1} - \beta^{'2}_{g+1} \right |.
\]
Indeed, the former chain of inequalities can be deduced from the position of the vertices $\alpha_f^1$ and $\alpha_f^2$ within the dual graphs of $\nu_{t_1}$ and $\nu_{t_2}$, and the fact that the continued fraction given by the value $\beta^{'1}_{g+1}$ (respectively, $\beta^{'2}_{g+1}$, $\beta^{'f}_{g+1}$) determines the subgraph $\Gamma_{g+1}$ of the dual graph of $\nu_{t_1}$ (respectively, $\nu_{t_2}$, $C_f$).



{\it v) If both valuations are in case c),} the result follows easily, since
\[
\left|\nu_{t_1}^N (f) - \nu_{t_2}^N (f)\right| = \left|\frac{\nu_{t_1}(f)}{\bar{\beta}_0^{1}} - \frac{\nu_{t_2}(f)}{\bar{\beta}_0^{2}}\right| =
\left|\frac{e^1_g \bar{\beta}_{g+1}^{f}}{\bar{\beta}_0^1} - \frac{e^2_g
\bar{\beta}_{g+1}^{f}}{\bar{\beta}_0^2}\right| = 0,
\]
which finishes our proof.
\end{proof}

From the above proposition, it can be deduced that the map $\mathcal{V}_{qm} \rightarrow \mathbb{R}$, given by $\nu \mapsto \hat{\mu}^N(\nu)$, is continuous, where $\mathcal{V}_{qm}$ is the tree of divisorial, curve and irrational valuations endowed with the strong topology of \cite{fj}. This fact will not be used in this paper.

\section{A Nagata-type conjecture for divisorial and irrational valuations}
\label{secNagata}

In this section, we consider divisorial and irrational valuations of the fraction field of $R= \mathcal{O}_{\mathbb{P}^2,p}$ centered at $R$, and use the notations given in Section \ref{minimal}.


\subsection{Very general valuations}

Let $\nu$ be as above and let $\Gamma$ be its dual graph. Consider the set ${\mathcal V}_{\Gamma}$ of plane valuations whose dual graph is $\Gamma$. Notice that the Hamburger-Noether expansions of valuations in ${\mathcal V}_{\Gamma}$ have the same \emph{structure}, that is, the same values $g, s_1, s_2, \ldots, s_g, k_1, k_2, \ldots,k_g, h_0, h_1, \ldots, h_{s_g}, \ldots$ and the same number and size of free rows (they differ only in the coefficients $a_{s_i k_i}$). In fact different choices of coefficients give rise to different valuations in ${\mathcal V}_{\Gamma}$. Since the coefficients $a_{s_i k_i}$, $1\leq i\leq g-1$, and $a_{s_g k_g}$, whenever its corresponding row is free, must be different from 0, the set ${\mathcal V}_{\Gamma}$ can be identified with the set of closed points of a Zariski open subset ${\mathcal U}_{\Gamma}$ of the affine space $\mathbb{A}_k^{b(\Gamma)}$, where $b(\Gamma):=h_0+\sum_{j=1}^{g'}(h_j-k_j+1)$, $g'$ being equal to $g$ (respectively, $g-1$) if the row whose left-hand side equals $w_{s_g-1}$ is free (respectively, otherwise).


\begin{pro}\label{yyy}
Let $\Gamma$ be the dual graph of a divisorial or irrational valuation. Then, for each positive integer $d$, the function ${\mathcal V}_{\Gamma}\rightarrow \mathbb{R}$, given by $\nu \mapsto \mu_d(\nu)$, gives rise (via the above identification) to an upper semicontinuous function $\phi: {\mathcal U}_{\Gamma} \rightarrow \mathbb{R}$.
\end{pro}

\begin{proof}
It is enough to prove that, for any $\alpha \in \mathbb{R}$, the preimage  $\phi^{-1}[\alpha,+\infty)$ is Zariski-closed.



Consider $\mathbf{a} \in {\mathcal U}_{\Gamma}$ and $\nu_\mathbf{a}$ its corresponding valuation. As mentioned, the Hamburger-Noether expansion of $\nu_\mathbf{a}$ provides parametric equations $u= u_{\mathbf{a}}(t,z), v= v_{\mathbf{a}}(t,z)$, where $u_{\mathbf{a}}(t,z),v_{\mathbf{a}}(t,z)\in k[\![t,z]\!]$ (respectively, $u= u_{\mathbf{a}}(t), v= v_{\mathbf{a}}(t)$, where $u_{\mathbf{a}}(t),v_{\mathbf{a}}(t)\in k\langle t \rangle\!$), when $\nu$ is divisorial (respectively, irrational). These equations allow us to compute the valuation $\nu(f)$ for any $f \in R$ as follows: $\nu(f)= \mathrm{ord}_t [f (u_{\mathbf{a}}(t,z), v_{\mathbf{a}}(t,z))]$ (respectively, $\nu(f)= \mathrm{ord}_t [f (u_{\mathbf{a}}(t), v_{\mathbf{a}}(t))]$) in case $\nu$ is divisorial (respectively, irrational).



Since the series involved in the above parameterizations have a finite number of non-zero terms, there exists a real number $M$ (that depends only on $\Gamma$ and $d$) such that the series which can be obtained by expressions $f (u_{\mathbf{a}}(t,z), v_{\mathbf{a}}(t,z))$ (respectively, $f (u_{\mathbf{a}}(t), v_{\mathbf{a}}(t))$), where $f$ is a polynomial in $k[u,v]$ whose degree is not larger than $d$ and $\mathbf{a} \in {\mathcal U}_{\Gamma}$, belong to the finite-dimensional vector subspace $Q$ of $k[\![t,z]\!]$ (respectively, $k\langle t \rangle$) generated by monomials of degree less than or equal to $M$.

To finish, consider the map $\psi: {\mathcal U}_{\Gamma}\times k[u,v]_d\rightarrow Q$  defined by $\psi(\mathbf{a},f)=f(u_{\mathbf{a}}(t,z),v_{\mathbf{a}}(t,z))$  (respectively,~$\psi(\mathbf{a},f)=f(u_{\mathbf{a}}(t),v_{\mathbf{a}}(t))$) if $\Gamma$ corresponds to a divisorial (respectively, irrational) valuation. The set $\{(\mathbf{a},f)\in {\mathcal U}_{\Gamma}\times k[u,v]_d\mid \nu_{\mathbf{a}}(f)\geq \alpha \}$ is Zariski-closed because it coincides with the pre-image $\psi^{-1}(S)$, by $\psi$, of the set $S:=\{\gamma \in Q \mid {\rm ord}_t(\gamma)\geq \alpha\}$ which is, clearly, a Zariski-closed subset of $Q$. Since $\psi^{-1}(S)$ is closed under scalar multiplication on the second component, it determines a Zariski-closed subset of ${\mathcal U}_{\Gamma}\times \mathbb{P}(k[u,v]_d)$, whose projection to ${\mathcal U}_{\Gamma}$ is $\{\mathbf{a}\in {\mathcal U}_{\Gamma}\mid \mu_d(\nu_{\mathbf{a}})\geq \alpha\}=\phi^{-1}([\alpha,+\infty))$, which concludes the proof.
\end{proof}

\begin{de}\label{vgd}
{\rm
Let $\Gamma $ be dual graph of a divisorial or irrational valuation. We will say that a property $S$ holds for a general (respectively, very general) valuation in ${\mathcal V}_{\Gamma}$ if there exists a finite (respectively, countable) collection of Zariski-open subsets of ${\mathcal U}_{\Gamma}$ such that $S$ holds for every valuation in the intersection of those sets. 
}
\end{de}
For simplicity, when in a statement we say that \emph{a divisorial or irrational valuation $\nu$ is general (respectively, very general)}, we mean that certain property given in the statement (which is implicitly understood) holds for a general (respectively, very general) valuation in ${\mathcal V}_{\Gamma}$, where $\Gamma$ is the dual graph of $\nu$.

An immediate consequence of Proposition \ref{yyy} is the following corollary.

\begin{cor}\label{uuu}
Let $\Gamma$ be the dual graph of a divisorial or irrational valuation $\nu$. Then $\hat{\mu}^N(\nu)$ takes its smallest value for very general valuations in ${\mathcal V}_{\Gamma}$.
\end{cor}

\subsection{The conjecture}

Let $({\mathcal C}, \bold{r})$ be a finite weighted configuration over $\mathbb{P}^2$, that is, a pair such that ${\mathcal C}$ is a finite configuration of infinitely near points of $\mathbb{P}^2$ and $\bold{r}$ is a map that assigns, to each point $p\in {\mathcal C}$, a non-negative integer $r_p$ called its multiplicity. Assume also that the weighted configuration is \emph{consistent}, that is, $r_p\geq \sum_{q\rightarrow p} r_q$ for all $p\in {\mathcal C}$. Consider also the ideal sheaf
$${\mathcal H}_{({\mathcal C}, \bold{r})}:=\pi_* \co_{X} \left( -\sum_{p\in {\mathcal C}} r_p E^*_p \right),$$
where $\pi:X\rightarrow \mathbb{P}^2$ denotes the composition of the blowing-ups centered at the points of ${\mathcal C}$ and each $E^*_p$ denotes the total transform on $X$ of the exceptional divisor given by the blowing-up centered at $p$.

If $d$ is a large enough positive integer, it holds that
$$
h^0(\gp^2,\co_{\gp^2}(d)\otimes {\mathcal H}_{({\mathcal C}, \bold{r})})=\frac{(d+1)(d+2)}{2}-\sum_{p\in {\mathcal C}} \frac{r_p(r_p+1)}{2}.
$$


The following conjecture is proposed in \cite{greuel} and predicts that, when the weighted configuration is general enough, the dimension of the space of global sections of the sheaf $\co_{\gp^2}(d)\otimes {\mathcal H}_{({\mathcal C}, \bold{r})}$ has the expected value (given by the above formula).

\begin{conj}\label{conj1}

Assume that $k=\mathbb{C}$. Let $({\mathcal C}, \bold{r})$ be a consistent finite weighted configuration of the projective plane. Suppose that $({\mathcal C}, \bold{r})$ is general among all the weighted configurations with the same proximities and let $d$ be an integer which is larger than the sum of the three biggest multiplicities involved in ${\mathcal C}$. Then
$$
h^0 \left(\gp^2,\co_{\gp^2}(d)\otimes {\mathcal H}_{({\mathcal C}, \bold{r})}\right)=\max \left\{0,\frac{(d+1)(d+2)}{2}-\sum_{p\in {\mathcal C}} \frac{r_p(r_p+1)}{2}\right\}.
$$
\end{conj}

Our purpose now is to prove that the above conjecture implies the following one for very general valuations:



\begin{conj}\label{conj2}
Assume that $k=\mathbb{C}$. If $\nu$ is a very general divisorial or irrational valuation such that $\left[{\rm vol}^N(\nu)\right]^{-1}\geq 9$,  then $\nu$ is minimal.
\end{conj}

\begin{rem}\label{quasi}
{\rm In \cite{d-h-k-r-s} it is proved that Conjecture \ref{conj1} implies the same statement as in Conjeture \ref{conj2} but, only, for what the authors call \emph{quasi-monomial valuations}; in our terminology, these are those valuations $\nu$ such that: either $\nu$ is irrational and the number $g+2$ of its Puiseux exponents is $2$, or it is divisorial and either $g+2=2$, or $g+2=3$ and $\nu$ is defined by a satellite divisor (notice that this notion of quasi-monomial valuations does not coincide with the one given in \cite{fj}).
}
\end{rem}


Given a finite configuration ${\mathcal C}=\{p_i\}_{i=1}^s$ as above, consider the sequence of blowing-ups $\pi:X\rightarrow \mathbb{P}^2$ centered at the points of $\mathcal C$ and denote by $E_i^*$ the total transform on $X$ of the exceptional divisor $E_i$ of the blowing-up centered at $p_i$. For each  $i$, $1 \leq i \leq s $, let $D_i:=\sum_{j=1}^s m_{ij} E_j^*$ be the divisor on $X$ given by the sequence of values $\{m_{ij}\}_{j=1}^i$ of the valuation defined by the divisor $E_i$ and $m_{ij}=0$ whenever $j>i$. Consider also the $s\times s$-matrix $G_{\mathcal C}=(g_{ij})$, defined by $g_{ij}=-9D_i\cdot D_j-(K_X\cdot D_i)(K_X\cdot D_j)$ for all $i,j\in \{1, 2, \ldots,s\}$, where $K_X$ denotes the canonical divisor.

\begin{de}
{\rm
A finite configuration $\mathcal C$ is \emph{P-sufficient} if $\mathbf{x} G_{\mathcal C} \mathbf{x}^t>0$ for all $\mathbf{x} \in \mathbb{R}^s\setminus \{0\}$ with non-negative coordinates.
}
\end{de}

The cone of curves of a surface $X$ obtained by blowing-up at the points of a P-sufficient configuration is (finite) polyhedral \cite[Theorem 2]{hel}. For proving that Conjecture \ref{conj1} implies Conjecture \ref{conj2}, we will use the following weaker concept.

\begin{de}
{\rm
A finite configuration $\mathcal C$ as above is \emph{almost P-sufficient} if $\mathbf{x} G_{\mathcal C} \mathbf{x}^t\geq 0$ for all $\mathbf{x} \in \mathbb{R}^s$ with non-negative coordinates.
}
\end{de}

An unibranch finite configuration $\mathcal C$ is P-sufficient if and only if the $(s,s)$-entry of the matrix $G_{\mathcal C}$ is strictly positive. That is to say, if and only if $9\sum_{j=1}^s m_j-(\sum_{j=1}^s m_j)^2>0$, where, for simplicity, we have set $m_j = m_{js}$. This fact was proved in \cite[Lemma 2]{hel} and an analogous proof allows us to show that, $\mathcal C$ unibranch is almost P-sufficient if and only if $9\sum_{j=1}^s m_j-(\sum_{j=1}^s m_j)^2\geq 0$.


\begin{lem}\label{lll1}
Assume that $k=\mathbb{C}$. Let $\nu$ be a very general divisorial valuation, let $\{\betabarra_i\}_{i=0}^{g+1}$ be its maximal contact values and assume that the configuration ${\mathcal C}_{\nu}$ is almost P-sufficient. If Conjecture \ref{conj1} holds then $\hat{\mu}(\nu)\leq 3\betabarra_0$.
\end{lem}

\begin{proof}
Suppose that ${\mathcal C}_{\nu}=\{p_i\}_{i=1}^s$ and let $\{m_i\}_{i=1}^s$ be the sequence of values of $\nu$. Let $d$ be a large enough positive integer. Take a polynomial $f(u,v)$ of degree  $d$ such that $f(0,0)=0$ and denote by $C_f$ the projective curve of $\gp^2$ of degree $d$ defined by $f$. For each $i\in \{1, 2, \ldots,s\}$, let $r_i$ be the multiplicity of the strict transform of $C_f$ at $p_i$.

Let $X_{\nu}$ be the surface obtained after blowing-up the points in ${\mathcal C}_{\nu}$. Denote by $\tilde{C}_f$ the strict transform of $C_f$ in $X_{\nu}$, then $\tilde{C}_f^2=d^2-\sum_{i=1}^s r_i^2$. Since ${\mathcal C}_{\nu}$ is almost P-sufficient, reasoning as in the proof of \cite[Lemma 2]{gm}, we deduce that $9\sum_{i=1}^s r_i^2-(\sum_{i=1}^s r_i)^2\geq 0$, and so
$$9(d^2-\tilde{C}_f^2)\geq \left(\sum_{i=1}^s r_i\right)^2.$$
Then
\begin{equation}\label{qq}
\left[ \nu(f)\right]^2=\left(\sum_{i=1}^s m_i r_i\right)^2\leq \betabarra_0^2\left(\sum_{i=1}^s r_i\right)^2\leq 9\betabarra_0^2 \left(d^2-\tilde{C}_f^2\right).
\end{equation}
By Conjecture \ref{conj1} we have that
$$\frac{(d+1)(d+2)}{2}-\sum_{i=1}^s \frac{r_i(r_i+1)}{2}\geq 1,$$
that is equivalent to
$$\tilde{C}_f^2\geq \sum_{i=1}^s r_i-3d.$$
In particular, $\tilde{C}_f^2\geq -3d$ and, then, by (\ref{qq}):
$$\left[\nu(f)\right]^2\leq 9\betabarra_0^2(d^2+3d).$$
This implies that $\mu_d(\nu)\leq 3\betabarra_0\sqrt{d^2+3d}$ (for every positive integer $d$) and therefore $\hat{\mu}(\nu)\leq 3\betabarra_0$.
\end{proof}

We will also use the following straightforward result.
\begin{lem}\label{lll2}
Let $\nu$ be a divisorial valuation, and  $d$ and $\alpha$  positive integers such that the cohomology $H^0(\gp^2,\co_{\gp^2}(d)\otimes {\mathcal P}_{\alpha})$ vanishes. Then,  $\mu_d(\nu)<\alpha$.
\end{lem}
\begin{proof}
On the contrary, suppose that $\mu_d(\nu)\geq \alpha$. Then there exists a polynomial $f(x,y)$ of degree less than or equal to $d$ such that $\nu(f)\geq \alpha$. But this defines a non-zero global section of $\co_{\gp^2}(d)\otimes {\mathcal P}_{\alpha}$, a contradiction.
\end{proof}

\begin{teo}
Conjecture \ref{conj1} implies Conjecture \ref{conj2}.
\end{teo}

\begin{proof}

Consider a curve valuation $\delta$ whose first centers are those in $C_\nu$ up to its last free point (which is included) and the remaining centers are free.
By Proposition \ref{proposiciondos}, one can suppose that $\left[{\rm vol}^N(\nu)\right]^{-1} > 9$ and $\nu$ is divisorial.



Let $\{m_i\}_{i=1}^s$ be the sequence of values of $\nu$ and let  $\{\betabarra_i\}_{i=0}^{g+1}$ be its sequence of maximal contact values. For each $k\in \mathbb{N}$, define $d_k:=\lfloor k(\betabarra_{g+1})^{1/2}\rfloor$ and consider the weighted configuration ${\mathcal K}_{\nu,k}:=({\mathcal C}_{\nu}, (k m_i)_{i=1}^s)$.
Notice that $d_k> 3k\betabarra_0$ for $k$ large enough. If $\alpha_k:=k\betabarra_{g+1}$, by \cite[Lemma 2.5]{d-h-k-r-s} one has that ${\mathcal P}_{\alpha_k}=\pi_*\co_{X_{\nu}}(-\sum_{i=1}^s km_iE^*_i)$ and, therefore, ${\mathcal P}_{\alpha_k}={\mathcal H}_{{\mathcal K}_{\nu,k}}$, where $\pi:X_{\nu}\rightarrow \mathbb{P}^2$ is the composition of the blowing-ups centered at the points in ${\mathcal C}_{\nu}$.

Assume that $k$ is large enough. If we prove that $h^0(\gp^2,\co_{\gp^2}(d_k)\otimes {\mathcal P}_{\alpha_k})=0$ then, by Lemma \ref{lll2}, we get that $\mu_{d_k}(\nu)<\alpha_k$ and, therefore,
$$\hat{\mu}(\nu)=\lim_{k\rightarrow \infty} \frac{\mu_{d_k}(\nu)}{d_k}\leq \lim_{k\rightarrow \infty} \frac{\alpha_k}{d_k}=\sqrt{\betabarra_{g+1}},$$
concluding the equality $\hat{\mu}(\nu)=\sqrt{\betabarra_{g+1}}$.

Hence, it only remains to show that the inequality
 $h^0(\gp^2,\co_{\gp^2}(d_k)\otimes {\mathcal P}_{\alpha_k})\geq 1$ cannot hold. Reasoning by contradiction, since $3k\betabarra_0$ is an upper bound of the sum of the three largest multiplicities in $\{k m_i\}_{i=1}^s$, by Conjecture \ref{conj1}, it holds that:
$$\frac{(d_k+1)(d_k+2)}{2}-\sum_{i=1}^s \frac{km_i(km_i+1)}{2}\geq 1.$$
Taking into account that $d_k\leq (\betabarra_{g+1})^{1/2}k$ and that $\sum_{i=1}^s m_i^2=\betabarra_{g+1}$, it is straightforward to check that the left hand side of the above inequality is less than or equal to
$$\frac{1}{2}\left( 3k\sqrt{\betabarra_{g+1}}-k\sum_{i=1}^s m_i+2 \right)$$
and, therefore,
$$3\sqrt{\betabarra_{g+1}}-\sum_{i=1}^s m_i\geq 0,$$
which implies that
$$
9\sum_{i=1}^s m_i^2-\left(\sum_{i=1}^s m_i \right)^2\geq 0.
$$


This shows that ${\mathcal C}_{\nu}$ is an almost P-sufficient configuration. Therefore, since  $\hat{\mu}(\nu)\leq 3\betabarra_0$ by Lemma \ref{lll1}, we get that
$$
\sqrt{\betabarra_{g+1}} \leq \hat{\mu}(\nu)\leq 3\betabarra_0
$$
and, as a consequence, $\left[{\rm vol}^N(\nu)\right]^{-1}\leq 9$, which gives the desired contradiction.
\end{proof}

\section{Families of minimal valuations} \label{VERY}

In this section, we consider a large family of divisorial valuations $\nu$  over the projective plane for which we are able, on the one hand, to compute explicitly the value $\hat{\mu}(\nu)$, and on the other hand, to determine a subset of minimal valuations. These facts, via semicontinuity, give rise to a large family of very general divisorial valuations for which an upper bound of $\hat{\mu}(\nu)$ is provided, and a subset of minimal very general valuations is obtained. Every mentioned family and subfamily contains valuations with arbitrarily many Puiseux exponents.


Throughout this section, we will use the notation as in Section \ref{minimal}: $p=(1:0:0)\in \mathbb{P}^2$ in projective coordinates $(X:Y:Z)$, $(u=Y/X,v=Z/X)$ are local coordinates around $p$ and $(x=X/Z, y=Y/Z)$ are coordinates in the affine chart defined by $Z\neq 0$; also, the line $L$ defined by $Z=0$ will be called the \emph{line at infinity}. Notice that $v=0$ is a local equation of $L$ at $p$.

\begin{de}
\label{NON}
A plane divisorial valuation $\nu$, of the quotient field $K$ of the ring $R=\mathcal{O}_{\mathbb{P}_k^2,p}$ and centered at $R$, is said to be non-positive at infinity whenever $\nu(f) \leq 0$ for any element $f \in k[x,y] \setminus \{0\}$, where $(x,y)$ are affine coordinates in the chart $Z\neq 0$.
\end{de}

\begin{rem}
{\rm Notice that if a plane divisorial valuation $\nu$ as above is non-positive at infinity, then either it is the $\mathfrak{m}$-adic valuation (where $\mathfrak{m}$ is the maximal ideal of $R$) or $\nu(v)>\nu(u)$.

}
\end{rem}

Valuations non-positive at infinity have an easy characterization:

\begin{pro} \label{DeltaD}
Let $\nu$ be a plane divisorial valuation of $K$ centered at $R$ which is not the $\mathfrak{m}$-adic one. Let $\betabarra_{g+1}$ be the last maximal contact value of $\nu$. Then $\nu$ is non-positive at infinity if, and only if, $\nu(v)>\nu(u)$ and the inequality $\nu(v)^2 \geq \betabarra_{g+1}$ holds.
\end{pro}

The above result was proved in \cite[Theorem 1]{cox}, where  good geometrical properties of surfaces defined by non-positive valuations are proved and it is characterized when the Cox ring of those surfaces is finitely generated.

The following result provides the value of $\hat{\mu}$ (see Section \ref{minimal}) for valuations non-positive at infinity.


\begin{pro} \label{51}
If $\nu$ is a valuation non-positive at infinity, then $\hat{\mu}(\nu)=\nu(v)$.
\end{pro}

\begin{proof}
We are going to show that, for any positive integer $d$, it holds that $\mu_d (\nu) = d \nu (v)$, which concludes the proof.
Indeed, let $g(u,v)$ be a polynomial in $k[u, v]$ of degree $d$. Then $g(u,v)=g(y/x,1/x)$ and
there exists
a polynomial $\tilde{g}(x, y) \in k[x, y]$ such that
$$
g(u,v)=\frac{\tilde{g}(x,y)}{x^d}.
$$
Therefore, on the one hand, $\nu (g)=\nu(\tilde{g})-d\nu(x)=\nu(\tilde{g})+d\nu(v)\leq d\nu(v)$, where the inequality holds because $\tilde{g}\in k[x,y]$ and $\nu$ is non-positive at infinity and, on the other hand,  $\nu(v^d)=d\nu(v)$, proving that $\mu_d(\nu)=d\nu(v)$ and the result.
\end{proof}

As a consequence of the above proposition and  semicontinuity (Corollary \ref{uuu}), we get an upper bound for $\hat{\mu}$ for a wide class of very general divisorial plane valuations.

\begin{teo}\label{zxz}
Let $\nu$ be a very general divisorial plane valuation (which is not the $\mathfrak{m}$-adic one), let $\{\betabarra_i\}_{i=0}^{g+1}$ be its sequence of maximal contact values and assume that $\betabarra_1^2\geq \betabarra_{g+1}$. Then $\hat{\mu}(\nu)\leq \betabarra_1$. If, in addition, the set $A_{\nu}:=\{a\in \mathbb{Z}\mid 1\leq a\leq \lfloor \betabarra_1/\betabarra_0 \rfloor \mbox{ and } a^2\betabarra_0^2\geq \betabarra_{g+1} \}$ is not empty, then $\hat{\mu}(\nu)\leq \betabarra_0\cdot \min(A_{\nu})$.
\end{teo}

\begin{proof}
Let $\nu$ be a plane valuation satisfying the conditions of the statement and let $\Gamma$ be its dual graph.

Assume first that $A_{\nu}$ is empty. Let $\nu'$ be a valuation in ${\mathcal V}_{\Gamma}$ (see Section \ref{secNagata} for the definition) such that the strict transforms of the line at infinity pass through all the initial free points in ${\mathcal C}_{\nu'}$. Then $\nu'(v)=\betabarra_1$ and, therefore, $\nu'$ is non-positive at infinity (by Proposition \ref{DeltaD}). Then $\hat{\mu}(\nu')=\betabarra_1$ by Proposition \ref{51} and, using semicontinuity (Corollary \ref{uuu}), it holds that $\hat{\mu}(\nu)\leq \hat{\mu}(\nu')=  \betabarra_1$.

It remains to consider the case when $A_{\nu}$ is not empty. Then there exists a valuation $\nu''\in {\mathcal V}_{\Gamma}$ such that the strict transforms of the line at infinity pass through the first $\min(A_{\nu})$ free points in ${\mathcal C}_{\nu''}$. Moreover $\nu(v)=\betabarra_0\cdot \min(A_{\nu})$ and, therefore, $\nu''$ is non-positive at infinity (by Proposition \ref{DeltaD}). Reasoning as before, $\hat{\mu}(\nu)\leq \hat{\mu}(\nu'')=\betabarra_0\cdot \min(A_{\nu})$.
\end{proof}

The following result is a direct consequence of Propositions \ref{DeltaD} and \ref{51}; it shows how to construct minimal valuations, non-positive at infinity. For stating it, we say that a plane divisorial valuation of $K$ centered at $R$, $\nu_2$, {\it enlarges} another valuation $\nu_1$ whenever the inclusion $\mathcal{C}_{\nu_1} \subseteq \mathcal{C}_{\nu_2}$, of their corresponding sequences of infinitely near points, holds.

\begin{pro} \label{MM}
Let be a plane valuation $\nu$ non-positive at infinity  (not the $\mathfrak{m}$-adic one) and denote by $\betabarra_{g+1}$ its last maximal contact value. Then, any divisorial valuation $\nu'$ enlarging $\nu$ and such that ${\mathcal C}_{\nu'} \setminus {\mathcal C}_{\nu}$ consists of $ \nu(v)^2  - \betabarra_{g+1}$ free points  is minimal.
\end{pro}

\begin{rem}
{\rm Notice that the $\mathfrak{m}$-adic valuation is minimal.

}
\end{rem}

\begin{pro}
\label{FF}
There is no minimal plane valuation $\nu$ non-positive at infinity defined by a satellite divisor.
\end{pro}

\begin{proof}
Let $\{\betabarra_0, \betabarra_1, \ldots, \betabarra_{g+1}\}$ be the sequence of maximal contact values of a minimal non-positive at infinity plane divisorial valuation $\nu$ defined by a satellite divisor. It is clear that either $\nu(v) = a \betabarra_0$ for some positive integer $a$ or $\nu(v)= \betabarra_1$, depending on the infinitely near points in ${\mathcal C}_\nu$ through which the strict transforms of the line at infinity $L$ pass. Since $\nu$ is defined by a satellite divisor, one gets  $\betabarra_{g+1}= n_g \betabarra_g$, where $n_g= \gcd (\betabarra_0, \betabarra_1, \ldots, \betabarra_{g-1})$.

Now, when $\nu(v) = a \betabarra_0$, from the fact that $\nu$ is minimal, we obtain $  a^2 \betabarra_0^2 =n_g \betabarra_{g}$ and so
\[
\left( \frac{a^2 \betabarra_0}{n_g} \right) \betabarra_0 = \betabarra_g,
\]
which is a contradiction because $a^2  \betabarra_0/n_g$ is a positive integer and $\betabarra_g$ does not belong to the semigroup generated by $\{\betabarra_0, \betabarra_1, \ldots, \betabarra_{g-1}\}$.

When $\mu(v) = \betabarra_1$, we also get a contradiction since $\betabarra_1^2 =n_g \betabarra_{g}$ is not possible because $(\betabarra_1/n_g) \betabarra_1 = \betabarra_g$ and $\betabarra_1/n_g$ is a positive integer.
\end{proof}

The next theorem provides a wide family of minimal very general divisorial valuations, and its proof is straightforward by semicontinuity (Corollary \ref{uuu}) from Proposition \ref{MM}.

\begin{teo}\label{vg}
Let $\nu$ and $\nu'$ be as in Propositon \ref{MM} and let $\Gamma_{\nu'}$ be the dual graph associated to $\nu'$. Then, very general valuations in ${\mathcal V}_{\Gamma_{\nu'}}$ are minimal.
\end{teo}






Now, we will give an explicit description of the family of valuations provided in the preceding theorem. For this purpose, fix a plane divisorial valuation $\omega$ (which is not the $\mathfrak{m}$-adic one) and let $\{\betabarra_i\}_{i=0}^{g+1}$ be its sequence of maximal contact values. Let $\Gamma_{\omega}$ be its dual graph and, for each positive integer $k$, denote by ${\Gamma}(k)$ the dual graph of any divisorial valuation $\eta$ whose configuration ${\mathcal C}_{\eta}$ has cardinality $k$ and all its points are free.

We define a new dual graph, $\Gamma_{\omega,k}$, obtained in the following way: Re-label the vertices of $\Gamma_{\omega}$ adding $k$ to each label, and join with an edge the last vertex of ${\Gamma}(k)$ and the root of $\Gamma_{\omega}$. $\Gamma_{\omega,k}$ is the obtained rooted tree after considering the root of ${\Gamma}(k)$ as the new root (see Figure \ref{grafo}). Define $\Gamma_{\omega,0}:=\Gamma_{\omega}$.

\begin{figure}[h]
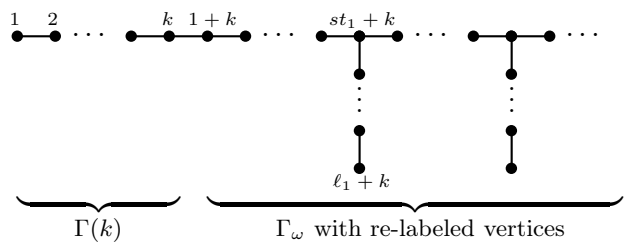


\begin{center}
\setlength{\unitlength}{0.5cm}%
\begin{Picture}(-6,0)(10,8)
\thicklines

\xLINE(-1,6)(0,6)
\Put(-1,6){\circle*{0.3}}
\Put(-1.2,6.3){\tiny{$k$}}
\xLINE(-2,6)(-1,6)
\Put(-2,6){\circle*{0.3}}
\Put(-4,6){$\;\;\ldots\;\;$}
\Put(-4,6){\circle*{0.3}}
\Put(-4.2,6.3){\tiny{$2$}}
\xLINE(-5,6)(-4,6)
\Put(-5,6){\circle*{0.3}}
\Put(-5.2,6.3){\tiny{$1$}}

\xLINE(0,6)(1,6)
\Put(0,6){\circle*{0.3}}
\Put(1,6){\circle*{0.3}}
\put(1,6){$\;\;\ldots\;\;$}
\xLINE(3,6)(4,6)
\Put(3,6){\circle*{0.3}}
\Put(4,6){\circle*{0.3}}
\xLINE(4,6)(4,5)
\Put(4,5){\circle*{0.3}}
\Put(3.9,4){$\vdots$}
\xLINE(4,3.5)(4,2.5)
\Put(4,3.5){\circle*{0.3}}
\Put(4,2.5){\circle*{0.3}}

\put(3.3,2){\tiny $\ell_1+k$}

\Put(3.2,6.3){\tiny{$st_1+k$}}
\Put(-0.5,6.3){\tiny{$1+k$}}



\xLINE(4,6)(5,6)
\Put(4,6){\circle*{0.3}}
\Put(5,6){\circle*{0.3}}
\Put(5,6){$\;\;\ldots\;\;$}
\xLINE(7,6)(8,6)
\Put(7,6){\circle*{0.3}}
\Put(8,6){\circle*{0.3}}
\xLINE(8,6)(8,5)
\Put(8,5){\circle*{0.3}}
\Put(7.9,4){$\vdots$}
\xLINE(8,3.5)(8,2.5)
\Put(8,3.5){\circle*{0.3}}
\Put(8,2.5){\circle*{0.3}}





\xLINE(8,6)(9,6)
\Put(8,6){\circle*{0.3}}
\Put(9,6){\circle*{0.3}}
\put(9,6){$\;\;\ldots\;\;$}









\Put(0,1.8){$\underbrace{\;\;\;\;\;\;\;\;\;\;\;\;\;\;\;\;\;\;\;\;\;\;\;\;\;\;\;\;\;\;\;\;\;\;\;\;\;\;\;\;\;\;\;\;\;\;\;\;\;\;\;}$}

\Put(1.8,0.7){\footnotesize $\Gamma_{\omega}$ with re-labeled vertices}

\Put(-5,1.8){$\underbrace{\;\;\;\;\;\;\;\;\;\;\;\;\;\;\;\;\;\;\;\;}$}

\Put(-3.5,0.7){\footnotesize $\Gamma(k)$}

\end{Picture}
\end{center}
  \caption{Graph $\Gamma_{\omega,k}$}
  \label{grafo}
\end{figure}

Consider the real number $\iota(\omega):=\frac{1}{2\betabarra_0}\left[ \betabarra_0-2\betabarra_1+\sqrt{\betabarra_0^2-4\betabarra_0\betabarra_1+4\betabarra_{g+1}} \right]$. Also, for each non-negative integer $k\geq \iota(\omega)$, consider the set
$$B_{\omega,k}:=\left\{a\in \mathbb{Z}\mid \sqrt{\frac{\betabarra_{g+1}}{\betabarra_0^2}+k}\leq a < \frac{\betabarra_1}{\betabarra_0}+k\right\}\cup\left\{\frac{\betabarra_1}{\betabarra_0}+k\right\}$$
and, for each $a\in B_{\omega,k}$, define $\Gamma_{\omega,k}^a$ the graph obtained as follows: Add $n$ to the labels of all vertices of the graph ${\Gamma}((a^2-k)\betabarra_0^2-\betabarra_{g+1})$, where $n$ is the number of vertices of $\Gamma_{\omega,k}$, and join with an edge the last vertex of $\Gamma_{\omega,k}$ with the root of ${\Gamma}((a^2-k)\betabarra_0^2-\betabarra_{g+1})$. $\Gamma_{\omega,k}^a$ is the obtained rooted tree, whose root is the one of $\Gamma_{\omega,k}$.


\begin{teo}\label{rock} Keeping the preceding notations, it holds that for each divisorial valuation $\omega$ (which is not the $\mathfrak{m}$-adic one), for each non-negative integer $k\geq \iota(\omega)$ and for each value  $a\in B_{\omega,k}$, very general valuations in ${\mathcal V}_{\Gamma_{\omega,k}^a}$ are minimal.
\end{teo}

\begin{proof}

Let $\omega, k$ and $a$ be as in the statement. If $\nu$ is any valuation in ${\mathcal V}_{\Gamma_{\omega,k}}$, its sequence of maximal contact values is $$\left\{\betabarra_0, \betabarra_1+k\betabarra_0,\ldots,\betabarra_i+k\frac{\betabarra_0^2}{e_{i-1}},\ldots, \betabarra_g+k\frac{\betabarra_0^2}{e_{g-1}}, \betabarra_{g+1}+k\betabarra_0^2 \right\}.$$
This follows, by using Noether's formula, from the fact that each maximal contact value attached to $\nu$, $\betabarra_i^\nu$, coincides with the intersection multiplicity at $p$ between a general element of $\nu$ and an analytically irreducible germ of curve whose strict transform by the sequence $\pi$, given by $\nu$, becomes regular and transversal to the exceptional divisor $E_{l_i}$ (see Figure \ref{fig2}). Notice that we denote by $E_{l_{g+1}}$ the exceptional divisor defining $\nu$.

It holds that $\sqrt{\frac{\betabarra_{g+1}}{\betabarra_0^2}+k}\leq  \frac{\betabarra_1}{\betabarra_0}+k$ because $k\geq \iota(\omega)$. Assume first that
\begin{equation}\label{pop}
\sqrt{\frac{\betabarra_{g+1}}{\betabarra_0^2}+k}\leq a < \frac{\betabarra_1}{\betabarra_0}+k.
\end{equation}
By the second inequality in (\ref{pop}) and since $\betabarra_1 + k \betabarra_0$ is the second maximal contact value of $\nu$, ${\mathcal C}_{\nu}$ has, at least, $a+1$ initial free points. So, we can add the assumption that the strict transforms of the line at infinity $L$ pass exactly through the first $a$ (free) points in ${\mathcal C}_{\nu}$. This means that $\nu(v)=a\betabarra_0$ and, by the first inequality in (\ref{pop}), the inequality $\nu(v)^2\geq \betabarra_{g+1}+k\betabarra_0^2$ holds. Therefore $\nu$ is non-positive at infinity. Enlarging $\nu$, by adding to ${\mathcal C}_{\nu}$
 $$\nu(v)^2-(\betabarra_{g+1}+k\betabarra_0^2)=(a^2-k)\betabarra_0^2-\betabarra_{g+1}$$
free points, we obtain, by Proposition \ref{MM}, a minimal valuation $\nu'$, whose associated dual graph is $\Gamma_{\omega,k}^a$. As a consequence, by Corollary \ref{uuu}, very general valuations in ${\mathcal V}_{\Gamma_{\omega,k}^a}$ are minimal.

To finish we consider the remaining case, where $a=(\betabarra_1/\betabarra_0) + k $. Pick the valuation $\nu$ with the condition that the strict transforms of the line at infinity $L$ pass through all the initial free points in ${\mathcal C}_{\nu}$. Then $\nu(v)=\betabarra_1+k\betabarra_0$, and taking into account that $(\betabarra_1+k\betabarra_0)^2\geq \betabarra_{g+1}+k\betabarra_0^2$ if, and only if, $k\geq \iota(\omega)$, condition that is fulfilled, we conclude that $\nu$ is non-positive at infinity. Reasoning as before, we also deduce that very general valuations in ${\mathcal V}_{\Gamma_{\omega,k}^a}$ are minimal.


\end{proof}
The following result, already proved in \cite{d-h-k-r-s}, can be trivially derived from the preceding theorem.

\begin{cor}
If $\nu$ is a very general divisorial valuation such that ${\mathcal C}_{\nu}$ consists of $t^2$ free points for some $t\in \mathbb{Z}_{> 0}$, then $\nu$ is minimal.
\end{cor}

\begin{proof}
It follows by Theorem \ref{rock} after taking into account that the dual graph of $\nu$ is $\Gamma_{\nu,0}^t$.
\end{proof}

\begin{rem}
{\rm The instrumental valuations used in  the proof of Theorem \ref{rock} correspond with valuations $\nu'$ provided in Theorem \ref{vg}. In fact, they are exactly the same ones. Indeed, consider $\nu$ and $\nu'$ as in
Theorem \ref{vg} and
let $\{\betabarra_i\}_{i=0}^{g+1}$ be the sequence of maximal contact values of $\nu$. Since $\nu$ is non-positive at infinity, one has that $\nu(v)=a\betabarra_0$ and $\nu(v)^2\geq \betabarra_{g+1}$, where either $a$ is a positive integer such that $1\leq a< \betabarra_1/\betabarra_0$, or $a=\betabarra_1/\betabarra_0$. Then it is clear that the dual graph of $\nu'$ is $\Gamma_{\nu,0}^a$.
}
\end{rem}

\section{An asymptotic result}
\label{laseis}

By Proposition \ref{FF}, we cannot use valuations non-positive at infinity defined by satellite divisors for obtaining, via semicontinuity, minimal very general  valuations. However, the next theorem provides an asymptotic result concerning the values $\hat{\mu}^{N}(\nu)$, where $\nu$ is a very general quasi-monomial valuation in the sense of the definition given in \cite{d-h-k-r-s} (see Remark \ref{quasi}). Before stating it, notice that the dual graph of such a valuation $\nu$ is completely determined by the value $[{\rm vol}^{N}(\nu)]^{-1}$, which is $\betabarra_1/ \betabarra_0$, where  $\betabarra_0$ and $\betabarra_1$ are the first two maximal contact values of $\nu$. Therefore, for each positive real number $t$, we can denote by ${\Theta}_t$ the dual graph of a valuation $\nu$ as above such that $[{\rm vol}^{N}(\nu)]^{-1}=t$.


\begin{teo} \label{teo57}
Consider, for each positive real number $t$, a very general valuation $\nu_t\in {\mathcal V}_{{\Theta}_t}$. Then
\[
\lim_{t \rightarrow \infty} \frac{\hat{\mu}^N(\nu_t)}{\sqrt{t}} =1.
\]
\end{teo}
\begin{proof}

Notice that, by Proposition \ref{proposiciondos}, we can assume that $t$ runs over the set of rational numbers. So, let $t\geq 2$ be a positive rational number and let $\omega_t$ be a (divisorial) valuation in ${\mathcal V}_{\Theta_t}$ such that the strict transforms of the line at infinity pass exactly through the first $\left\lceil \sqrt{t} \right\rceil$ points of the configuration ${\mathcal C}_{\omega_t}$. This means that  $\omega_t(v)=\left\lceil \sqrt{\betabarra_1/\betabarra_0} \right\rceil \betabarra_0$, where $\betabarra_0$ and $\betabarra_1$ are the first two maximal contact values of $\omega_t$. Notice that $[\omega_t(v)]^2\geq t$ and, therefore, $\omega_t$ is non-positive at infinity. Hence $\hat{\mu}^N(\omega_t)=\left\lceil \sqrt{t} \right\rceil$ by Proposition \ref{51}, and thus  the quotient
\[
\frac{\hat{\mu}^N (\omega_t)}{\sqrt{t}} = \frac{\left\lceil \sqrt{t}\right\rceil}{\sqrt{t}}
\]
converges to $1$ when $t$ tends to infinity. Finally, by Corollary \ref{uuu} we have that
$$\frac{\hat{\mu}^N(\nu_t)}{\sqrt{t}}\leq \frac{\hat{\mu}^N(\omega_t)}{\sqrt{t}}$$
and the result follows because the quotient $\frac{\hat{\mu}^N(\nu_t)}{\sqrt{t}}$ is greater than or equal to 1.
\end{proof}

\section*{Acknowledgements}
The authors would like to thank A. K{\"u}ronya and J. Ro\'e for helpful comments.

\end{document}